\documentclass[review]{elsarticle}

\usepackage{mathrsfs,exscale,ifthen,amsfonts,amsmath,amscd,amssymb,amsthm,dsfont}
\usepackage[margin=1in]{geometry}

\DeclareMathOperator{\supp}{supp}

\DeclareMathOperator{\defeq}{\mathrel{\mathop:}=}

\newtheorem{corollary}{Corollary}[section]

\newtheorem{lemma}{Lemma}[section]

\newtheorem{proposition}{Proposition}[section]
\newtheorem{remark}{Remark}[section]

\newtheorem{theorem}{Theorem}[section]

\theoremstyle{remark}

\numberwithin{equation}{section}

\newcommand{\Ew}{\mathcal{E}^\omega}
\newcommand{\E}{\mathcal{E}}
\newcommand{\mean}{\mathbb{E}}
\newcommand{\PR}{\mathbb{P}}
\newcommand{\N}{\mathbb{N}}
\newcommand{\R}{\mathbb{R}}

\newcommand{\F}{\mathcal{F}}

\renewcommand{\phi}{\varphi}
\renewcommand{\epsilon}{\varepsilon}

\makeatletter
\def\namedlabel#1#2{\begingroup
   \def\@currentlabel{#2}%
   \label{#1}\endgroup
}
\makeatother

\begin{document}

\begin{frontmatter}

\title{Invariance Principle for symmetric Diffusions in a degenerate and unbounded stationary and ergodic Random Medium.}

\author[mymainaddress]{Alberto Chiarini\corref{cor1}\footnote{The first author is supported by RTG 1845.}}
\ead{chiarini@math.tu-berlin.de}

\author[mymainaddress]{Jean-Dominique Deuschel\corref{cor2}}
\ead{deuschel@math.tu-berlin.de}

\address[mymainaddress]{Department of Mathematics, Technische Universit\"at zu Berlin, Stra\ss e des 17. Juni 136, 10623 Berlin, Germany}

\begin{abstract} We study a symmetric diffusion $X$ on $\R^d$ in divergence form in a stationary and ergodic environment, with measurable unbounded and degenerate coefficients $a^\omega$. The diffusion is formally associated with $L^\omega u = \nabla\cdot(a^\omega\nabla u)$, and we make sense of it through Dirichlet forms theory. We prove for $X$ a quenched invariance principle, under some moment conditions on the environment; the key tool is the sublinearity of the corrector obtained by Moser's iteration scheme. \\
\\
Nous \'etudions une diffusion sym\'etrique $X$ sur $R^d$ en forme de
divergence dans un environnement al\'eatoire stationnaire et ergodique, dont les
coefficients $a^\omega$ sont mesurables et d\'eg\'en\'er\'es. Cette diffusion qui est formellement
engendr\'ee par l'op\'erateur $L^\omega u=\nabla\cdot(a^\omega\nabla u)$, peut \^etre
d\'efinie \`a l'aide de la th\'eorie des formes de Dirichlet. Nous d\'emontrons pour $X$ un principe d'invariance presque s\^ur sous des conditions de
moment de l'environnement; l'outil crucial est la sous-lin\'earit\'e du
correcteur obtenu \`a l'aide de l' it\'eration introduite par J. Moser.
\end{abstract}


\end{frontmatter}

\section{Description of the Main Result}

We are interested in the study of reversible diffusions in a random environment. Namely, we are given an infinitesimal generator $L^\omega$ in divergence form
\begin{equation}\label{eq:generator}
 L^\omega u(x) = \nabla \cdot ( a^\omega (x)\nabla u(x)),\quad x\in\R^d
\end{equation}
where $a^\omega(x)$ is a symmetric $d$-dimensional matrix depending on a parameter $\omega$ which describes a random realization of the environment.

We model the environment as a probability space $(\Omega,\mathcal{G},\mu)$ on which a measurable group of transformations $\{\tau_x\}_{x\in\R^d}$ is defined. One may think of $\tau_x\omega$ as a translation of the environment $\omega\in\Omega$ in the direction $x\in\R^d$. The random field $\{a^\omega(x)\}_{x\in\R^d}$ will then be constructed simply by taking a random variable $a:\Omega\to\R^{d\times d}$ and by defining $a^\omega(x)\defeq a(\tau_x \omega)$, we will often use the notation $a(x;\omega)$ for $a^\omega(x)$ as  well.
We assume that the random environment $(\Omega,\mathcal{G},\mu)$, $\{\tau_x\}_{x\in\R^d}$ is stationary and ergodic. A precise formulation of the setup is given in section 3.

It is well known that when $x\to a^\omega(x)$ is bounded and uniformly elliptic, uniformly in $\omega$, then a quenched invariance principle holds for the diffusion process $X_t^\omega$ associated with $L^\omega$. This means that, for $\mu$-almost all $\omega\in\Omega$, the scaled process $X^{\epsilon,\omega}_t\defeq \epsilon X^{\omega}_{t/\epsilon^2}$ converges in distribution to a Brownian motion with a non-trivial covariance structure as $\epsilon$ goes to zero; this is known as diffusive limit. See for example the classic result of Papanicolau and Varadhan \cite{papvar} where the coefficients are assumed to be differentiable, and \cite{Osada1983} for measurable coefficients and more general operators.

Recently, a lot of efforts has been put into extending this result beyond the uniform elliptic case. For example \cite{fannjiang1997} consider a non-symmetric situation with uniformly elliptic symmetric part and unbounded antisymmetric part and the recent paper \cite{BaMathieu} proves an invariance principle for divergence form operators $L u= e^V \nabla\cdot(e^{-V} \nabla u)$ where $V$ is periodic and measurable. They only assume that $e^V+e^{-V}$ is locally integrable. For what concerns ergodic and stationary environment a recent result has been achieved in the case of random walk in random environment in \cite{deuschelslowikandresharnack}, \cite{deuschelslowikandres}. In these works moments of order greater than one are needed to get an invariance principle in the diffusive limit; \cite{deuschelslowikandres} and the techniques therein are the main inspiration for our paper.

The aim of our work is to prove a quenched invariance principle for an operator $L^\omega$ of the form \eqref{eq:generator} with a random field $a^\omega(x)$ which is ergodic, stationary and possibly unbounded and degenerate.
Denote by $a:\Omega\to\R^{d\times d}$ the $\mathcal{G}$-measurable random variable which describes the field through $a^\omega(x)=a(\tau_x\omega)$. We assume that $a$ is symmetric and that there exist $\Lambda,\lambda$, $\mathcal{G}$-measurable, positive and finite, such that:
\begin{itemize}
 \item[$(a.1)$\namedlabel{ass:a.1}{$(a.1)$}] for $\mu$-almost all $\omega\in\Omega$ and all $\xi\in\R^d$
\[
 \lambda(\omega)|\xi|^2\leq \langle a(\omega) \xi,\xi\rangle \leq \Lambda(\omega)|\xi|^2;
\]
 \item[$(a.2)$\namedlabel{ass:a.2}{$(a.2)$}] there exist $p,q\in[1,\infty]$ satisfying $1/p+1/q<2/d$ such that
\[
 \mean_\mu [\lambda^{-q}]<\infty,\quad\mean_\mu [\Lambda^{p}]<\infty,
\]
 \item[$(a.3)$\namedlabel{ass:a.3}{$(a.3)$}] as functions of $x$, $\lambda^{-1}(\tau_x\omega),\Lambda(\tau_x\omega)\in L^\infty_{loc}(\R^d)$ for $\mu$-almost all $\omega\in\Omega$.
\end{itemize}
Since $a^\omega(x)$ is meant to model a random field, it is not natural to assume its differentiability in $x\in\R^d$. Accordingly, the operator defined in \eqref{eq:generator} does not make any sense, and the techniques coming from Stochastic differential equations and It\^o calculus are not very helpful neither in constructing the diffusion process, nor in performing the relevant computation.

The theory of Dirichlet forms is the right tool to approach the problem of constructing a diffusion. Instead of the operator $L^\omega$ we shall consider the bilinear form obtained by $L^\omega$, formally integrating by parts, namely
\begin{equation}\label{eq:df}
\Ew(u,v)\defeq \sum_{i,j}\int_{\R^d} a^\omega_{ij}(x)\partial_i u(x)\partial_j v(x) dx
\end{equation}
for a proper class of functions $u,v\in\F^\omega\subset L^2(\R^d,dx)$, more precisely $\F^\omega$ is the closure of $C_0^\infty(\R^d)$ in $L^2(\R^d,dx)$ with respect to $\E+(\cdot,\cdot)_{L^2}$. It is a classical result of Fukushima \cite[Theorem 7.2.2]{fukushima1994dirichlet} and \cite[Ch. II example 3b]{rockner} that it is possible to associate to \eqref{eq:df} a diffusion process $\{X^\omega, \PR_x^\omega, x\in \R^d\}$ as soon as $(\lambda^\omega)^{-1}$ and $\Lambda^\omega$ are locally integrable. It is well known that there is a properly exceptional\footnote{A set $\mathcal{N}\subset \R^d$ is called properly exceptional if $\mathcal{N}$ is Borel, it has Lebesgue measure zero, and $\PR_x(X_t\in\mathcal{N}\mbox{ or }X_{t-}\in\mathcal{N} \mbox{ for some }t\geq0)=0$ for all $x\in \R^d\setminus \mathcal{N}$.
} set $\mathcal{N}^\omega\subset \R^d$ of $X^\omega$ such that
the associated process is uniquely determined up to the ambiguity of starting points in $\mathcal{N}^\omega$, in our situation the set of exceptional points may depend on the realization of the environment. Assumption \ref{ass:a.3} is designed to remove the ambiguity about the properly exceptional set $\mathcal{N}^\omega$. We will then prove that assumption \ref{ass:a.2} and ergodicity of the environment are enough to grant that the process $X^\omega$ starting from any $x\in \R^d$ does not explode for almost all realization of the environment.

\begin{remark}Moment conditions on the environment are a very natural assumption in order to achieve a quenched invariance principle for symmetric diffusions, indeed at least the first moment of $\Lambda$ and $\lambda^{-1}$ is required to obtain the result. As a counterexample one can consider a periodic environment, namely the $d$-dimensional torus $\mathbb{T}^d$, and the following generator in divergence form
	\[
	L f(x)\defeq \frac{1}{\phi(x)}\nabla\cdot (\phi(x)\nabla f(x)),
	\]
where $\phi:\mathbb{T}^d\to \R$ is defined by $\phi(x)\defeq 1_B(x) |x|^{-d}+1_{B^c}(x)$ being $B\subset \mathbb{T}^d$ a ball of radius one centered in the origin. It is clear that $\phi^\alpha\in L^1(\mathbb{T}^d)$ for all $\alpha<1$ but not for $\alpha=1$. If we look for example to $d=2$, then the radial part of the process associated to $L$, for the radius less than one, will be a Bessel process with parameter $\delta=0$ which is known to have a trap in the origin.
\end{remark}

\begin{remark} As observed in the previous remark, if we want to prove an invariance principle, dealing with symmetric diffusions forces the degeneracy of the diffusion coefficient not to be too strong. Namely, the diffusion coefficient can eventually be zero only on a set of null Lebesgue measure. On the other hand, in the case of non-symmetric diffusions the diffusion coefficient is allowed to vanish in open sets, as was proved in the periodic environment by \cite{Hairer20082462} and further extended and generalized in \cite{delaruerhodes}, \cite{sow2009homogenization}, \cite{pardoux2011homogenization}. In these works the strong degeneracy of the diffusion coefficient is compensated by the drift through the H\"ormander's condition; as a result and in contrast with our setting, the coefficients need to be smooth enough.
\end{remark}

Once the diffusion process $X^\omega$ is constructed, the standard approach to diffusive limit theorems consists in showing the weak compactness of the rescaled process and in the identification of the limit. In the case of bounded and uniformly elliptic coefficients the compactness is readily obtained by the Aronson-Nash estimates for the heat kernel. In order to identify the limit, we use the standard technique used in \cite{fannjiang1997}, \cite{kozlov1985} and \cite{Osada1983}; namely, we decompose the process $X_t^\epsilon$ into a martingale part, called the \emph{harmonic coordinates} and a fluctuation part, called the \emph{correctors}. The martingale part is supposed to capture the long time asymptotic of $X_t^\epsilon$, and will characterize the diffusive limit.

The challenging part is to show that the correctors are uniformly small for almost all realization of the environment, this is attained generalizing Moser's arguments \cite{moser1964} to get a maximal inequality for positive subsolutions of uniformly elliptic, divergence form equations. In this sense the relation $1/p+1/q<2/d$ is designed to let the Moser's iteration scheme work.  This integrability assumption firstly appeared in \cite{edmundspeletier} in order to extend the results of De Giorgi and Nash to degenerate elliptic equations.  A similar condition was also recently exploited in \cite{Zhikov} to obtain estimates of Nash - Aronson type for solutions to degenerate parabolic equations.  They look to  generator of the form $\mathcal{L} u = \partial_t u - e^{-V}\nabla \cdot ( e^V \nabla u)$,
with the assumption that $\sup_{r\geq 1}|r|^{-d}\int_{|x|\leq r} e^{pV}+e^{-qV} dx <\infty$.

We want to stress out that condition \ref{ass:a.3} is needed to prove neither the sublinearity of the corrector nor its existence, we used it only to have a more regular density of the semigroup associated to $X^\omega$ and avoid some technicalities due to exceptional sets in the framework of Dirichlet form theory.

Once the correctors are shown to be sublinear, the standard invariance principle for martingales \cite{helland1982} gives the almost sure convergence to the Wiener measure.
\begin{theorem}\label{thm:invprinc} Assume \ref{ass:a.1}, \ref{ass:a.2} and \ref{ass:a.3} are satisfied. Let $\mathbf{M}^\omega\defeq (X_t^\omega, \PR_x^\omega)$, $x\in\R^d$, be the minimal diffusion process associated to $(\Ew,\F^\omega)$ on $L^2(\R^d,dx)$. Then the following hold
\begin{itemize}
\item[(i)] For $\mu$-almost all $\omega\in\Omega$ the limits
\[
\lim_{t\to\infty} \frac{1}{t}\mean_0^\omega[X^\omega_t(i) X^\omega_t(j)] =\mathbf{d}_{ij}\quad i,j=1,...,d
\]
exist and are deterministic constants.
\item[(ii)] For $\mu$-almost all $\omega\in\Omega$, the laws of the processes $X_t^{\omega,\epsilon}\defeq \epsilon X^\omega_{t/\epsilon^2}$, $\epsilon>0$ over $C([0,+\infty),\R^d)$ converge weakly as $\epsilon\to 0$ to a Wiener measure  having the covariance matrix equal to  $\mathbf{D}=[\mathbf{d}_{ij}]$. Moreover $\mathbf{D}$ is a positive definite matrix.
\end{itemize}
\end{theorem}

\paragraph{Description of the method} One of the main objective of the paper is to show that the correctors $\chi  = (\chi^1,\dots,\chi^d):\R^d\times \Omega \to \R^d$ are locally sublinear, namely that
\[
\limsup_{\epsilon\to 0} \sup_{|x|\leq R} \epsilon|\chi(x/\epsilon,\omega)| = 0,\quad \forall R>0,\,\mu\mbox{-a.s.}
\]

To obtain a priori estimates on the correctors $\chi$ we exploit the fact that they are constructed in such a way that they are solutions of a Poisson's equation, which is formally given by
\begin{equation}\label{eq:correctoreq}
\nabla \cdot (a^\omega(x) \nabla \chi^k(x,\omega)) = \nabla \cdot (a^\omega(x) \nabla \pi^k(x)),
\end{equation}
where $\pi^k(x) := x_k$ is the projection to the $k$th-coordinate.

The equation above has been studied extensively and generalized in many directions, also beyond the linear case, for an introduction, see for example the monographs \cite{evans2010partial}, \cite{gilbarg2001elliptic} and  for recent developments in the theory see \cite{heinonen2006nonlinear}. When the matrix $a^\omega$ is uniformly elliptic and bounded, uniformly in $\omega\in\Omega$, namely if
\[
c^{-1}|\xi|^2 \leq  \langle a^\omega(x)\xi,\xi \rangle \leq c |\xi|^2
\]
for some $c\geq 1$, it is natural to look for weak solutions to \eqref{eq:correctoreq} in the classical Sobolev space of square integrable functions with square integrable weak derivatives. It is a classical result due to Moser \cite{moser1964} that an elliptic Harnack inequality holds and a result from Nash \cite{nash1958} and De Giorgi \cite{zbMATH03138423} that solutions are H\"older continuous.

The situation changes dramatically when the coefficients are degenerate. In the most typical situation there is a positive weight $\theta:\R^d\to \R$ and a constant $c>1$ such that
\[
\theta(x)|\xi|^2 \leq  \langle a^\omega(x)\xi,\xi \rangle \leq c\,\theta(x)|\xi|^2.
\]
In this setting one looks for solutions to equation \eqref{eq:correctoreq} in the weighted Sobolev space $W^{1,2}(\R^d,\theta)$ which is the set of weakly differentiable functions $u:\R^d\to \R$ such that
\[
\int_{\R^d} |u|^2 \theta dx<\infty,\mbox{ and } \int_{\R^d} |\nabla u|^2 \theta dx<\infty,
\]
we refer to \cite{heinonen2006nonlinear}, \cite{zhikov1998weighted} for more information on weighted Sobolev spaces. It was shown in \cite{fabes1982local} that in order to have local regularity of solutions to \eqref{eq:correctoreq} it is enough to have weights which are volume doubling, namely such that there exists a constant $C>0$ for which
\[
\int_{B_{2R}(x)} \theta(y)\,dy \leq C \int_{B_R(x)}\theta(y)\,dy,\quad \forall R>0,\,\forall x\in \R^d,
\]
and which satisfy weighted Sobolev and Poincar\'e inequalities. This weights are known in general as $p$-admissible (See \cite{heinonen2006nonlinear}), but for our discussion of the linear operator $L^\omega=\nabla\cdot(a^\omega\nabla\,)$ it is enough to look at $2$-admissible weights.

\begin{remark}\label{rem:voldoubling}In our setting it is not possible to expect the volume doubling property for small balls. The ergodic theorem ensures only that for all $x\in \R^d$ and $\mu$-almost all $\omega\in\Omega$ there exist $R_0^\omega(x)>0$ and a dimensional constant $C>0$ such that for all $R>R_0^\omega(x)$
	\[
	\int_{B_{2R}(x)} \Lambda^\omega(y)\,dy \leq C \int_{B_R(x)}\Lambda^\omega(y)\,dy,
	\]
	being $B_R(x)$ the ball of center $x$ and radius $R$.
	We remark that the constant $R_0^\omega(x)$ cannot be taken uniformly in $x\in \R^d$, and $\sup_{x\in \R^d} R_0^\omega(x)$ may be infinite.
\end{remark}

Examples of $2$-admissible weights are the functions in the Muckenhaupt's class $A_2$, we refer to \cite{fabes1982local}, \cite{heinonen2006nonlinear}, \cite{torchinsky2012real} and to the original research paper \cite{muckenhoupt1972weighted} for an exhaustive treatment on the subject. Here we briefly recall that  the class $A_2$ is the set of all non negative functions $\theta:\R^d\to [0,\infty]$ for which there exists a constant $C>0$ such that
\begin{equation}\label{eq:muckenhaupt}
\sup_{R>0} \sup_{x\in\R^d} \biggr(\frac{1}{|B_R(x)|}\int_{B_R(x)} \theta(y)\,dy\biggl)
\biggr(\frac{1}{|B_R(x)|}\int_{B_R(x)} \theta^{-1}(y)\,dy\biggl) \leq C.
\end{equation}
It is well known that weights in the class $A_2$ are volume doubling and satisfy a weighted Sobolev inequality. To be more precise, denote by $\theta(B):=\int_B\theta dx$, then there exist constants $C,\delta>0$ such that for all $1\leq k\leq d/(d-1)+\delta$
\begin{equation}\label{ineq:muck}
\biggl(\frac{1}{\theta(B)}\int_{B}|u|^{2k} \theta dx\biggr)^\frac{1}{k}\leq C |B|^\frac{2}{d} \frac{1}{\theta(B)}\int_{B}|\nabla u|^{2} \theta dx\quad \biggl( \leq C |B|^\frac{2}{d} \frac{\E(u,u)}{\theta(B)}\biggr)
\end{equation}
being $B$ any ball in $\R^d$ and $u\in C_0^\infty(B)$.

Working with admissible weights has the advantage of being able to state H\"older continuity results for weak solutions to \eqref{eq:correctoreq}. It is still an open problem to identify the optimal conditions that a weight has to satisfy in order to grant continuity of weak solutions, see the survey paper \cite{cavalheiro2008weighted} for details.

Many authors relied on Muckenhaupt's classes and weighted Sobolev spaces to prove homogenization results. We quote \cite{de1992homogenization} for the periodic case and \cite{engstrom2006homogenization} for the ergodic case. In the latter the weights are assumed to belong to a Muckenhaupt class for almost all the realizations of the environment.

In our paper, to prove the sublinearity of the corrector, we assume that the coefficient $a^\omega(x)$ satisfies
\[
\lambda^\omega(x)|\xi|^2 \leq  \langle a^\omega(x)\xi,\xi \rangle \leq \Lambda^\omega(x)|\xi|^2,\quad \mu\mbox{-a.s.}
\]
and $\mean_\mu[\lambda^{-q}]$, $\mean_\mu[\lambda^{-q}]<\infty$ with $1/p+1/q<2/d$. In this case, the weights $\lambda^\omega(x) := \lambda(\tau_x\omega)$ and $\Lambda^\omega(x):= \Lambda(\tau_x\omega)$ do not belong to any of the classes mentioned above, since, as explained in Remark \ref{rem:voldoubling}, in general the measures $\lambda^\omega(x)dx$ and $\Lambda^\omega(x)dx$ are not volume doubling. The ergodicity of the environment and the fact that $\mean_\mu[\lambda^{-1}],\mean_\mu[\Lambda]$ are finite ensure only that
\[
\sup_{x\in \R^d}\limsup_{R\to\infty}\frac{1}{|B_R(x)|}\int_{B_R(x)} \frac{1}{\lambda^\omega(y)}\,dy<\infty,\quad\sup_{x\in \R^d}\limsup_{R\to\infty}\frac{1}{|B_R(x)|}\int_{B_R(x)}\Lambda^\omega(y) dy <\infty,
\]
$\mu$-almost surely, and, contrary to \eqref{eq:muckenhaupt}, it is not possible to interchange the supremum and the limit staying finite.

Another characterizing feature of our model is that we don't assume $\Lambda^\omega\leq c \lambda^\omega$. We cannot expect regularity for solutions to \eqref{eq:correctoreq}, however, we show that the ergodicity of the environment and the moment conditions \ref{ass:a.2} are enough to obtain the sublinearity of the correctors; this is done in the same spirit of \cite{fannjiang1997} where an unbounded but uniformly bounded away from zero non-symmetric case is considered.

Moser's method to derive a maximal inequality for solutions to \eqref{eq:correctoreq} is based on two steps. One wants first to get a Sobolev inequality to control some $L^\rho$-norm in terms of the Dirichlet form and then control the Dirichlet form of any solution by a lower moment. This sets up an iteration which leads to control the supremum of the solution on a ball by a lower norm on a slightly bigger ball. In the uniform elliptic and bounded case
this is rather standard and it is possible to control the $L^{2d/(d-2)}$-norm of a solution by its $L^2$-norm through the classical Sobolev inequality. In the case of Muckenhaupt's weights the iteration can be set using the Sobolev inequality \eqref{ineq:muck} on the weighted Sobolev space.

In our paper we are able to control locally on balls the $\rho$-norm of a solution by its $2 p^*$-norm, with
$\rho = 2qd/(q(d-2)+d)$ and $p^*=p/(p-1)$. For the Moser iteration we need $\rho>2 p^*$ which is equivalent to condition $1/p+1/q<2/d$. Indeed, by means of H\"older's inequality and the standard Sobolev inequality, for a ball $B$ of radius $R>0$ and center $x\in \R^d$, we can write
\[
\biggl(\frac{1}{|B_R(x)|}\int_{B_R(x)}|u|^{\rho/p^*}\Lambda^\omega dy\biggr)^\frac{2 p^*}{\rho}\leq C(\lambda,\Lambda,x,R)|B_R(x)|^\frac{2}{d}\frac{\E(u,u)}{|B_R(x)|} .
\]
where
\[
C(\lambda,\Lambda,x,R) := C(d) \biggl(\frac{1}{|B_R(x)|}\int_{B_R(x)} (\lambda^\omega(y))^{-q}dy\biggr)^\frac{1}{q}\biggl( \frac{1}{|B_R(x)|}\int_{B_R(x)} (\Lambda^\omega(y))^{p}dy\biggr)^\frac{2}{\rho(p-1)},
\]
being $C(d)>0$ a constant depending only on the dimension.
The Sobolev inequality above must be compared with \eqref{ineq:muck}. In opposition to \eqref{ineq:muck}, the constant in front of the inequality is strongly dependent on $x\in\R^d$ and $R>0$.
Therefore, the estimates we derive in Section 2 to control the Dirichlet form of a solution by its $2p/(p-1)$-norm, although following from very well established arguments, are a necessary step in order to clarify the dependence of the constants on
\[
\frac{1}{|B_R(x)|}\int_{B_R(x)} (\lambda^\omega(y))^{-q}dy,\quad \frac{1}{|B_R(x)|}\int_{B_R(x)} (\Lambda^\omega(y))^{p}dy.
\]
The maximal inequality which we obtain in Section 2.3 behaves nicely in the scaling limit, due to the ergodic theorem,  and is enough to state the sublinearity of the corrector.

\begin{remark} It is believed that the optimal condition for a quenched invariance principle to hold is $\mean_\mu[\lambda^{-1}]$, $\mean_\mu[\Lambda]<\infty$. In periodic environment this has been proven recently in \cite{BaMathieu} using ideas coming from harmonic analysis and Muckenhaupt's weights. The authors consider a generator in divergence form given by $L u= e^V \nabla\cdot(e^{-V} \nabla u)$, where $V:\R^d\to \R$ is periodic and measurable such that $e^V+e^{-V}$ is locally integrable. Their argument relies on a time change and on the Sobolev inequality
\[
\biggl(
\int_{\mathbb{T}^d}|u|^{r} w\, dx\biggr)^\frac{2}{r}\leq C \int_{\mathbb{T}^d}|\nabla u|^{2} e^{-V} dx\quad \]
where $\mathbb{T}^d$ is the $d$-dimensional torus, $u\in C^1(\mathbb{T}^d)$ centered, $r>2$ and $w$ is expressed as an Hardy-Littlewood maximal function.

In this setting it is not possible to use Moser's iteration technique to prove the sublinearity of the corrector on balls, since to bound the right hand side by the $L^s(\mathbb{T}^d,w)$ norm for some $s<r$ would require further assumptions on the integrability of $e^V+e^{-V}$. In fact, they don't prove sublinearity of the correctors on balls but along the path of the process. This approach relies on a global uniform upper bound for the density of the process, which can be established due to the compactness of the periodic environment, and the fact that the process of the environment seen from the
particle is just the projection of the diffusion on the torus $\mathbb{T}^d$.
\end{remark}

\begin{remark} Under the conditions \ref{ass:a.1}, \ref{ass:a.2} and that a quenched invariance principle holds, Moser's method can be successfully applied to obtain a quenched local central limit theorem for the process associated to $(\E^\omega,\F^{\Lambda,\omega})$ on $L^2(\R^d,\Lambda^\omega dx)$, being $\F^{\Lambda,\omega}$ the closure of $(\E^\omega,C_0^\infty(\R^d))$ in $L^2(\R^d,\Lambda^\omega dx)$, see \cite{deuschelslowikandresharnack}, \cite{chiarinideuschelLCLT}. In these papers, the proof relies on a parabolic Harnack inequality, whose constant depends strongly on the space-time cylinder considered. Thus, it cannot be applied to obtain H\"older continuity of the density. Nevertheless, it is shown that in the diffusive limit it is possible to control oscillations by means of the ergodic theorem.
	
Despite the fact that a quenched invariance principle is believed to hold for $\mean_\mu[\lambda^{-1}]$, $\mean_\mu[\Lambda]<\infty$, it was shown in \cite{deuschelslowikandresharnack} that the condition $\mean_\mu[\lambda^{-q}]$, $\mean_\mu[\Lambda^p]<\infty$, with $1/p+1/q<2/d$ is sharp, for general stationary and ergodic random environment, for a quenched local central limit theorem to hold.
	
\end{remark}

A summary of the paper is the following. In Section \ref{sec:sobmos} we develop a priori estimates for solutions to elliptic equations, following Moser's scheme. In this section the random environment plays no role, and accordingly we have deterministic inequalities in a fairly general framework. Also, we construct a minimal diffusion process associated to the deterministic version of \eqref{eq:df} and we discuss its properties.

In Section \ref{sec:diffre} we apply the results obtained in Section \ref{sec:sobmos} to construct a diffusion process for almost all $\omega\in\Omega$, we define the environment process, and we show how to use it in order to prove that the diffusion is non-explosive.

In Section \ref{sec:corr} we prove the existence of the harmonic coordinates and of the corrector. In particular we prove that we can decompose our process in the sum of a martingale part, of which we can compute exactly the quadratic variation, and a fluctuation part.

In Section \ref{sec:proof} we use the results of the previous Sections in order to prove the sublinearity of the correctors and, given that, Theorem \ref{thm:invprinc}.

\section{Sobolev's inequality and Moser's iteration scheme}\label{sec:sobmos}

\subsection{Notation and Basic Definitions}

In this section we forget about the random environment. With a slight abuse of notation we will note with $a(x)$, $\lambda(x)$ and $\Lambda(x)$ the deterministic versions of $a(\tau_x\omega)$, $\lambda(\tau_x \omega)$ and $\Lambda(\tau_x \omega)$.

We are given a symmetric matrix $a:\R^d\to \R^{d\times d}$ such that
\begin{itemize}
\item[$(b.1)$\namedlabel{ass:b.1}{$(b.1)$}]there exist $\lambda,\Lambda:\R^d\to\R$ non-negative such that for almost all $x\in\R^d$ and $\xi\in\R^d$
 \[
  \lambda(x)|\xi|^2\leq \langle a(x) \xi,\xi\rangle \leq \Lambda(x)|\xi|^2,
 \]
\item[$(b.2)$\namedlabel{ass:b.2}{$(b.2)$}]there exist $p,q\in[1,\infty]$ satisfying  $1/p+1/q<2/d$ such that
 \[
  \sup_{r\geq 1} \frac{1}{|B_r|} \int_{B_r} \Lambda^p +\lambda^{-q} \,dx <\infty.
 \]
\end{itemize}
\begin{remark}
 By means of the ergodic theorem, \ref{ass:a.1} and \ref{ass:a.2} imply that the function $x\to a(\tau_x\omega)$ satisfies \ref{ass:b.1} and \ref{ass:b.2} for $\mu$-almost all $\omega\in\Omega$.
\end{remark}

\begin{remark}\label{rem:embedding} Let $B\subset\R^d$ be a ball. Assumptions \ref{ass:b.1} and $\ref{ass:b.2}$ imply that, for $u\in C_0^\infty(B)$,
\[
\|1_B\lambda^{-1}\|_q^{-1}\|\nabla u\|^2_{2q/q+1}\leq \int_{\R^d}\langle a \nabla u,\nabla u\rangle dx\leq \|1_B \Lambda\|_p \|\nabla u\|_{2p^*}^2,
\]
where $p^* = p/(p-1)$. The relation $1/p+1/q<2/d$ is designed in such a way that the Sobolev's conjugate of $2q/(q+1)$ in $\R^d$, which is given by
\begin{equation}\label{eq:rho}
\rho(q,d):= \frac{2qd}{q(d-2)+d},
\end{equation}
satisfies $\rho(q,d) > 2 p^*$, which implies that the Sobolev space $W^{1,2q/(q+1)}(B)$ is compactly embedded in $L^{2p^*}(B)$, see for example Chapter 7 in \cite{gilbarg2001elliptic}.
\end{remark}

Since the generator given in \eqref{eq:generator} is not well defined, in order to construct a process formally associated to it, we must exploit Dirichlet forms theory. We shall here present some basic definitions coming from the Dirichlet forms theory; for a complete treatment on the subject see \cite{fukushima1994dirichlet}.

Let $X$ be a locally compact metric separable space, and $m$ a positive Radon measure on $X$ such that $\supp[m]=X$. Consider the Hilbert space $L^2(X,m)$ with scalar product $\langle \cdot,\cdot\rangle$. We call a \emph{symmetric} form, a non-negative definite bilinear form $\E$ defined on a dense subset $\mathcal{D}(\E)\subset L^2(X,m)$. Given a symmetric form $(\E,\mathcal{D}(\E))$ on $L^2(X,m)$,  the form $\E_\beta\defeq \E+\beta \langle \cdot,\cdot\rangle$ defines itself a symmetric form on $L^2(X,m)$ for each $\beta>0$. Note that $\mathcal{D}(\E)$ is a pre-Hilbert space with inner product $\E_\beta$. If $\mathcal{D}(\E)$ is complete with respect to $\E_\beta$, then $\E$ is said to be \emph{closed}.

A closed symmetric form $(\E,\mathcal{D}(\E))$ on $L^2(X,m)$ is called a \emph{Dirichlet form} if it is Markovian, namely if for any given $u\in \mathcal{D}(\E)$, then $v=(0\vee u)\wedge1$ belongs to $\mathcal{D}(\E)$ and $\E(v,v)\leq \E(u,u)$.

We say that the Dirichlet form $(\E,\mathcal{D}(\E))$ on $L^2(X,m)$ is \emph{regular} if there is a subset  $\mathcal{H}$ of $\mathcal{D}(\E)\cap C_0(X)$ dense in $\mathcal{D}(\E)$ with respect to $\E_1$ and dense in $C_0(X)$ with respect to the uniform norm. $\mathcal{H}$ is called a \emph{core} for $\mathcal{D}(\E)$.

We say that the Dirichlet form $(\E,\mathcal{D}(\E))$ is \emph{local} if for all $u,v\in\mathcal{D}(\E)$ with disjoint compact support $\E(u,v)=0$. $\E$ is said \emph{strongly local} if $u,v\in \mathcal{D}(\E)$ with compact support and $v$ constant on a neighborhood of $\supp u$ implies $\E(u,v)=0$.
\\

Let $\theta:\R^d\to\R$ be a non-negative function such that $\theta^{-1},\theta$ are locally integrable on $\R^d$.
Consider the symmetric form $\E$ on $L^2(\R^d,\theta dx)$ with domain $C_0^\infty(\R^d)$ defined by
\begin{equation}\label{eq:DF}
\E(u,v)\defeq\sum_{i,j}\int_{\R^d} a_{ij}(x)\partial_i u(x)\partial_j v(x)\,dx.
\end{equation}
Then, $(\E,C_0^\infty(\R^d))$ is  \emph{closable} in $L^2(\R^d,\theta dx)$ thanks to \cite{rockner}[Ch. II example 3b], since $\lambda^{-1},\Lambda\in L^1_{loc}(\R^d)$ by \ref{ass:b.2}. We shall denote by $(\E,\F^\theta)$ such a closure; it is clear that $\F^\theta$ is the completion of $C_0^\infty(\R^d)$ in $L^2(\R^d,\theta dx)$ with respect to $\E_1$. If $u\in \F^\theta$, then $u$ is weakly differentiable with derivatives in $L^1_{loc}(\R^d)$ and
$\E(u,u)$ takes the form \eqref{eq:DF} with $\partial_i u$, $i=1,...,d$
being the weak derivative of $u$ in direction $i$. Observe that $(\E,\F^\theta)$ is a strongly local regular Dirichlet form, having $C_0^\infty(\R^d)$ as a core. In the case that $\theta\equiv 1$ we will simply write $\F$.

The Dirichlet forms theory \cite[Theorem 7.2.2]{fukushima1994dirichlet} allows to construct a diffusion process $\mathbf{M}^{\theta}\defeq(X_t^\theta,\PR^\theta_x,\zeta^\theta)$, associated to $(\E,\F^\theta)$,  starting from all points outside a properly exceptional set. Since we shall work with random media, the set of exceptional points may depend on the particular realization of the environment. In Section \ref{sec:mindiff} we shall construct a diffusion process starting for all $x\in\R^d$ at the price of local boundedness of the coefficients.

Fix a ball $B\subset\R^d$ and consider $\E$ as defined in \eqref{eq:DF} but on $L^2(B,\theta dx)$, and with domain $C_0^\infty(B)$, then clearly $(\E,C_0^\infty(B))$ is  closable in $L^2(B,\theta dx)$. We denote by $(\E,\F^\theta_B)$ the closure, which also in this case is a strongly local regular Dirichlet form.
\subsection{Sobolev's inequalities}

Let us introduce some notation. Let $B\subset \R^d$ be an open bounded set. For a function $u:\R^d\to\R$ and for $r\geq 1$ we note
\[\|u\|_{r} \defeq \biggl(\int_{\R^d} |u(x)|^r dx\biggr)^\frac{1}{r},\quad
\|u\|_{r,\Lambda} \defeq \biggl(\int_{\R^d} |u(x)|^r\Lambda(x)dx\biggr)^\frac{1}{r},\quad\|u\|_{B,r} \defeq \biggl(\frac{1}{|B|}\int_B |u(x)|^r\,dx\biggr)^\frac{1}{r}.
\]

In the next proposition it is enough to assume the local integrability of $\Lambda$ and the $q$-local integrability of $\lambda^{-1}$.

\begin{proposition}[local Sobolev inequality] Fix a ball $B\subset\R^d$. Then there exists a constant $C_{sob}>0$, depending only on the dimension $d\geq2$, such that for all $u\in\F_B$
\begin{equation}\label{eq:localsobolev}
\|u\|_\rho^2\leq C_{sob}\|1_B \lambda^{-1}\|_q\, \E(u,u).
\end{equation}
\end{proposition}
\begin{proof}  We start proving \eqref{eq:localsobolev} for $u\in C_0^\infty(B)$. Since $\rho$ as defined in \eqref{eq:rho} is the Sobolev conjugate of $2q/(q+1)$ in $\R^d$, by the classical Sobolev's inequality there exists $C_{sob}>0$ depending only on $d$ such that
\[
\|u\|_\rho\leq C_{sob} \|\nabla u\|_{2q/(q+1)},
\]
where it is clear that we are integrating over $B$. By H\"older's inequality and \ref{ass:b.1} we can estimate the right hand side as follows
\[
\|\nabla u\|_{2q/(q+1)}^2=\Bigl(\int_B|\nabla u|^\frac{2q}{q+1}\lambda^\frac{q}{q+1}\lambda^{-\frac{q}{q+1}}\,dx\Bigr)^\frac{q+1}{q}\leq \|1_B\lambda^{-1}\|_{q}\,\E(u,u),
\]
which leads to \eqref{eq:localsobolev} for $u\in C_0^\infty(B)$. By approximation, the inequality is easily extended to $u\in\F_B$.
\end{proof}

\begin{proposition}[local weighted Sobolev inequality] Fix a ball $B\subset\R^d$. Then there exists a constant $C_{sob}>0$, depending only on the dimension $d\geq2$, such that for all $u\in\F_B^\Lambda$
\begin{equation}\label{eq:localweightedsobolev}
\|u\|_{\rho/p^*,\Lambda}^2\leq C_{sob}\|1_B \lambda^{-1}\|_q\|1_B\Lambda\|_p^{2p^*/\rho}\, \E(u,u).
\end{equation}
\end{proposition}
\begin{proof} The proof easily follows from H\"older's inequality
\[
\|u\|_{\rho/p^*,\Lambda}^2\leq \|u\|_\rho^2 \|1_B\Lambda\|_p^{2p^*/\rho}
\]
and the previous proposition.
\end{proof}

\begin{remark}
From these two Sobolev's inequalities it follows that the domains $\F_B$ and $\F_B^\Lambda$ coincide for all balls $B\subset \R^d$. Indeed, from \eqref{eq:localsobolev} and \eqref{eq:localweightedsobolev}, since $\rho,\rho/p^*>2$, we get that $(\F_B,\E)$ and $(\F_B^\Lambda,\E)$ are two Hilbert spaces; therefore $\F_B,\F_B^\Lambda$ coincide with their extended Dirichlet space which by \cite[page 324]{Fukushima1987} is the same, hence $\F_B = \F_B^\Lambda$.
\end{remark}

\paragraph{Cutoffs}Since we want to get apriori estimates for solutions to elliptic partial differential equations in the spirit of the classical theory, we will need to work with functions that are locally in $\F$ or $\F^\Lambda$ and with cutoffs.

Let $B\subset\R^d$ be a ball, a cutoff on $B$ is a function $\eta\in C_0^\infty(B)$, such that $0\leq \eta\leq 1$.
Given $\theta:\R^d\to\R$ as before, we say that $u\in\F_{loc}^\theta$, if for all balls $B\subset\R^d$ there exists $u_B\in\F^\theta$ such that $u = u_B$ almost surely on $B$.

In view of these notations, for $u,v\in \F^\theta_{loc}$ we define the bilinear form
\begin{equation}\label{eq:dfcutoff}
\E_\eta(u,v)=\sum_{i,j}\int_{\R^d} a_{ij}(x)\partial_i u(x)\partial_j v(x)\eta^2(x)\,dx.
\end{equation}

\begin{lemma}\label{lem:cutoff} Let $B\subset \R^d$ and consider a cutoff $\eta\in C_0^\infty(B)$ as above. Then, $u\in\F_{loc}\cup \F^\Lambda_{loc}$ implies $\eta u\in \F_B$.
\end{lemma}
\begin{proof} Take $u\in\F_{loc}^\Lambda$, then there exists $\bar{u}\in\F^\Lambda$ such that $u=\bar{u}$ on $2B$. Let $\{f_n\}_\N\subset C_0^\infty(\R^d)$ be such that $f_n\to \bar{u}$ with respect to $\E+\langle\cdot,\cdot\rangle_\Lambda$. Clearly $\eta f_n\in\F_B^\Lambda$ and $\eta f_n\to \eta \bar{u}=\eta u$ in $L^2(B,\Lambda dx)$. Moreover
	\[
	\E(\eta f_n-\eta f_m)\leq2\E(f_n-f_m)+\|\nabla\eta\|_\infty^2\int_B|f_n-f_m|^2\Lambda dx.
	\]
	Hence $\eta f_n$ is Cauchy in $L^2(B,\Lambda dx)$ with respect to $\E+\langle\cdot,\cdot\rangle_\Lambda$, which implies that $\eta u\in \F^\Lambda_B = \F_B$. If $u\in\F_{loc}$ the proof is similar, and one has only to observe that $\{f_n\}$ is Cauchy in $W^{2q/(q+1)}(B)$, which by Sobolev's embedding theorem implies that $\{f_n\}$ is Cauchy in $L^2(B,\Lambda dx)$.
\end{proof}

\begin{proposition}[local Sobolev inequality with cutoff] Fix a ball $B\subset\R^d$ and a cutoff function $\eta\in C_0^\infty (B)$ as above. Then there exists a constant $C_{sob}>0$, depending only on the dimension $d\geq2$, such that for all $u\in\F_{loc}^\Lambda\cup\F_{loc}$
\begin{equation}\label{eq:localsobolev-cutoff}
\|\eta u\|_{\rho}^2\leq 2 C_{sob}\|1_B \lambda^{-1}\|_q \Bigl[\E_\eta(u,u)+\|\nabla \eta\|_\infty^2\|1_B u\|^2_{2,\Lambda}\Bigr],
\end{equation}
and
\begin{equation}\label{eq:localweightedsobolev-cutoff}
\|\eta u\|_{\rho/p^*,\Lambda}^2\leq 2 C_{sob}\|1_B \lambda^{-1}\|_q \|1_B\Lambda\|_p^{2p^*/\rho} \Bigl[\E_\eta(u,u)+\|\nabla \eta\|_\infty^2\|1_B u\|^2_{2,\Lambda}\Bigr].
\end{equation}
\end{proposition}
\begin{proof}
We prove only $\eqref{eq:localsobolev-cutoff}$, being \eqref{eq:localweightedsobolev-cutoff} analogous. Take $u\in \F_{loc}\cup\,\F^\Lambda_{loc}$, by Lemma \ref{lem:cutoff}, $\eta u\in\F_B$, therefore we can apply \eqref{eq:localsobolev} and get
\[
\|\eta u\|_{\rho}^2\leq C_{sob}\|1_B \lambda^{-1}\|_q\, \E(\eta u,\eta u).
\]
To get \eqref{eq:localsobolev-cutoff} we compute $\nabla(\eta u) = u \nabla \eta+\eta\nabla u$ and we easily estimate
\begin{align*}
\E(\eta u,\eta u) &= \int_{\R^d} \langle a \nabla(\eta u),\nabla(\eta u)\rangle dx\\&\leq 2\int_{\R^d} \langle a \nabla u,\nabla u\rangle \eta^2  dx + 2\int_{\R^d} \langle a \nabla \eta,\nabla \eta\rangle |u|^2  dx
\\&\leq2\E_\eta(u,u)+2\|\nabla\eta\|_\infty^2\|1_B u\|_{2,\Lambda}^2.
\end{align*}
\end{proof}
\subsection{Maximal inequality for Poisson's equation}

Let $f:\R^d\to\R$ be some function with essentially bounded  weak derivatives. We say that $u\in\F_{loc}$ is a  solution (subsolution or supersolution) of the Poisson equation, if
\begin{equation}\label{eq:Poisson}
 \E(u,\phi)= -\int_{\R^d} \langle a \nabla f, \nabla\phi \rangle dx\quad(\leq\mbox{ or }\geq)
\end{equation}
for all $\phi\in C_0^\infty(\R^d)$, $\phi\geq 0$. For a ball $B\subset\R^d$, we say that $u\in\F_{loc}$ is a solution (subsolution or supersolution) of the Poisson equation in $B$ if $\eqref{eq:Poisson}$ is satisfied for all  $\phi\in \F_B$, $\phi\geq 0$.

Given a positive subsolution $u\in\F_{loc}$ of \eqref{eq:Poisson}, we would like to test for $\phi=u^{2\alpha-1} \eta^2$ with $\alpha>1$ and $\eta$ a cutoff function in $B$. The aim is to get a priori estimates for $u$. One must be careful with powers of the function $u$. Indeed, in general $u^{2\alpha-1}$ is not a weakly differentiable function, and therefore it is not clear that $\phi\in \F$. The following Lemma is needed to address such a problem
\begin{lemma}\label{lemma:moser}Let $G:(0,\infty)\to (0,\infty)$ be a Lipschitz function with Lipschitz constant $L_G>0$. Assume also that $G(0+)=0$.
	Take $u\in \F$, $u\geq\epsilon$, for some $\epsilon>0$ then $G(u)\in \F$.
\end{lemma}
\begin{proof} The result follows observing that $G(u)/L_G$ is a normal contraction of $u\in \F$, and by standard Dirichlet form theory, see \cite[Ch. 1]{fukushima1994dirichlet} for details.
\end{proof}

\begin{proposition} Let $u\in \F_{loc}$ be a subsolution of \eqref{eq:Poisson} in $B$. Let $\eta\in C_0^\infty(B)$ be a cutoff function, $0\leq\eta\leq1$. Then there exists a constant $C_1>0$ such that for all $\alpha\geq 1$
\begin{equation}\label{eq:meanvalue}
\|\eta u^+\|_{B,\alpha\rho}^{2\alpha} \leq \alpha^2 C_1\|\lambda^{-1}\|_{B,q}\|\Lambda\|_{B,p}|B|^\frac{2}{d}\Bigl[\|\nabla\eta\|^2_\infty\| u^+\|^{2\alpha}_{B,2\alpha p_*}+\|\nabla f\|^2_\infty \|u^+\|^{2\alpha-2}_{B,2\alpha p_*}\Bigr].
\end{equation}
\end{proposition}
\begin{proof} We can assume $u\in \F_{2B}$ since we shall look only inside $B$ and $u\in\F_{loc}$. We build here a function $G$ to be a prototype for a power function. Let $G:(0,\infty)\to (0,\infty)$  be a piecewise $C^1$ function such that $G'(s)$ is bounded by a constant say $C>0$. Assume also that $G$ has a non-negative, non-decreasing derivative $G'(x)$ and $G(0+)=0$. Define $H(s)\geq 0$ by $H'(s)=\sqrt{G'(s)}$, $H(0+)=0$. Observe that we have $G(s)\leq sG'(s)$, $H(s)\leq sH'(s)$.
Let $\eta$ be a cutoff in $B$ as above. Then, we have by Lemma \ref{lemma:moser} and Lemma \ref{lem:cutoff} that
\[
\phi=\eta^2 (G(u^++\epsilon)-G(\epsilon))\in \F_{B}.
\]
In particular, $\phi$ is a proper test function. In order to lighten the notation we denote $G_\epsilon(x) \defeq G(x^++\epsilon)-G(\epsilon)$ and $H_\epsilon(x) \defeq H(x^++\epsilon)-H(\epsilon)$. Since $u$ is a subsolution to \eqref{eq:Poisson} in $B$, we have
\begin{equation}\label{eq:comp1}
\Ew(u,\eta^2 G_\epsilon(u))\leq -\int_{\R^d} \langle a \nabla f, \nabla (\eta^2 G_\epsilon(u)) \rangle dx.
\end{equation}
Consider first the left hand side and observe that
\[
\E(u,\eta^2 G_\epsilon(u)) =\int_{\R^d}\langle a \nabla u^+,\nabla u^+\rangle G'_\epsilon(u) \eta^2dx+2\int_{\R^d} \langle a \nabla u,\nabla \eta\rangle G_\epsilon(u)\eta dx.
\]
Since
\[
\int_{\R^d} \langle a \nabla u^+,\nabla u^+\rangle G_\epsilon'(u)\eta^2dx=\E_\eta(H_\epsilon(u),H_\epsilon(u)),
\]
moving everything on the right hand side of \eqref{eq:comp1}, and taking the absolute value, we have
\begin{equation}\label{eq:comp2}
\E_\eta(H_\epsilon(u),H_\epsilon(u)) \leq 2 \int_{\R^d}| \langle a \nabla u,\nabla \eta\rangle G_\epsilon(u)\eta| dx +\int_{\R^d} |\langle a \nabla f, \nabla (G_\epsilon(u)\eta^2) \rangle | dx.
\end{equation}
The first term is estimated using $G_\epsilon(u)\leq u^+ G'_\epsilon(u)$ and by Cauchy Schwartz inequality. (We use also the fact that $u^+\nabla u = u^+\nabla u^+$).
\begin{align*}
  \int_{\R^d} |\langle a \nabla u,\nabla \eta\rangle G_\epsilon(u)\eta |dx\leq & \int_{\R^d} |\langle a \nabla u^+,\nabla \eta\rangle G'_\epsilon(u) u^+\eta |dx \\ \leq&\E_\eta(H_\epsilon(u),H_\epsilon(u))^\frac{1}{2}\|G_\epsilon'(u)(u^+)^2\|_{1,\Lambda}^\frac{1}{2}\|\nabla \eta\|_\infty .
\end{align*}
The second term, after using Leibniz rule, is controlled by
\[
 \int_{\R^d} |\langle a\nabla f, \nabla u^+\rangle G_\epsilon'(u)\eta^2  |dx+ 2 \int_{\R^d} |\langle a\nabla f, G_\epsilon(u)\eta \nabla\eta \rangle|dx
\]
whose terms can be estimated by
\[
\int_{\R^d} |\langle a\nabla f, \nabla u^+\rangle G'_\epsilon(u)\eta^2  |dx \leq \|\nabla f\|_\infty \|1_B G'_\epsilon(u)\|_{1,\Lambda}^\frac{1}{2} \E_\eta(H_\epsilon(u),H_\epsilon(u))^\frac{1}{2}
\]
and by
\[
\int_{\R^d} |\langle a\nabla f,\nabla\eta \rangle G_\epsilon(u)\eta |dx \leq \|\nabla\eta\|_\infty \|\nabla f\|_\infty  \|G_\epsilon(u)1_B\|_{1,\Lambda}.
\]
Putting everything together in \eqref{eq:comp2} we end up with the estimate
\begin{align*}
\E_\eta(H_\epsilon(u),H_\epsilon(u)) &\leq 2 \|G_\epsilon'(u)(u^+)^2\|_{1,\Lambda}^\frac{1}{2}\|\nabla \eta\|_\infty\E_\eta(H_\epsilon(u),H_\epsilon(u))^\frac{1}{2}\\ &+ \|\nabla f\|_\infty \|1_B G'_\epsilon(u)\|_{1,\Lambda}^\frac{1}{2}  \E_\eta(H_\epsilon(u),H_\epsilon(u))^\frac{1}{2}
\\& +2\|\nabla\eta\|_\infty \|\nabla f\|_\infty  \|G_\epsilon(u)1_B\|_{1,\Lambda},
\end{align*}
which finally gives, up to a universal constant $c>0$,
\begin{align*}
\E_\eta(H_\epsilon(u),H_\epsilon(u))\leq & c\Bigl[ \|G'_\epsilon(u)(u^+)^2\|_{1,\Lambda}\|\nabla\eta\|^2_\infty+\|\nabla f\|^2_\infty \|1_B G_\epsilon'(u)\|_{1,\Lambda} \\&+\|\nabla\eta\|_\infty \|\nabla f\|_\infty  \|G_\epsilon(u)1_B\|_{1,\Lambda}\Bigr].
\end{align*}
At this point, it is important to observe that $H_\epsilon(u)\in \F$ so that we can apply the Sobolev's inequality \eqref{eq:localsobolev-cutoff} with cut-off function $\eta$, namely
\[
\|\eta H_\epsilon(u)\|^2_\rho\leq 2C_{sob}\|1_B\lambda^{-1}\|_{q} \Bigl[\Ew_\eta (H_\epsilon(u),H_\epsilon(u))+ \|\nabla \eta\|^2_\infty \|1_B H_\epsilon(u)\|^2_{2,\Lambda}\Bigr].
\]
Concatenating the two inequalities yields
\begin{multline*}
\|\eta H_\epsilon(u)\|^2_\rho\leq 2 C_1\|1_B\lambda^{-1}\|_{q}\Bigl[\|H_\epsilon'(u)^2 u^2\|_{1,\Lambda}\|\nabla\eta\|^2_\infty+\|\nabla f\|^2_\infty \|1_B H_\epsilon'(u)^2\|_{1,\Lambda} \\+\|\nabla\eta\|_\infty \|\nabla f\|_\infty  \|G_\epsilon(u)1_B\|_{1,\Lambda}+\|\nabla \eta\|^2_\infty \|1_B H_\epsilon(u)\|^2_{2,\Lambda}\Bigr]
\end{multline*}
Finally it is time to fix a $H,G$ as power-like function. Namely we take, for $\alpha>1$
\[
H_N(x)\defeq\begin{cases}x^\alpha & x\leq N\\
\alpha N^{\alpha-1}x+(1-\alpha)N^{\alpha} & x>N
\end{cases}
\]
which corresponds in taking
\[
G_N(x)=\int_0^x H_N'(s)^2\,ds.
\]
The function $G_N(x)$ has the right properties, moreover $H_N(x)\uparrow x^\alpha$ and $G_N(x)\uparrow \frac{\alpha^2}{2\alpha-1} x^{2\alpha-1}$ as $N$ goes to infinity. Therefore, letting $N\to\infty$, and using the monotone convergence theorem, we obtain
\begin{multline*}
\|\eta (u^++\epsilon)^\alpha\|^2_\rho\leq 2 C_1\|1_B\lambda^{-1}\|_{q}\Bigl[(\alpha^2+1)\| 1_B  (u^++\epsilon)^{2\alpha}\|_{1,\Lambda}\|\nabla\eta\|^2_\infty\\+\|\nabla f\|^2_\infty \alpha^2 \|1_B  (u^++\epsilon)^{2\alpha-2}\|_{1,\Lambda} +\frac{\alpha^2}{2\alpha-1}\|\nabla\eta\|_\infty \|\nabla f\|_\infty  \|u^{2\alpha-1}1_B\|_{1,\Lambda}\Bigr].
\end{multline*}
Taking the limit as $\epsilon\to 0$ and averaging over balls we get
\begin{multline*}
\|\eta (u^+)^\alpha\|^2_{B,\rho} \leq 2 C_1\|\lambda^{-1}\|_{B,q}\|\Lambda\|_{B,p}|B|^\frac{2}{d}\Bigl[(\alpha^2+1)\| (u^+)^{2\alpha}\|_{B,p_*}\|\nabla\eta\|^2_\infty\\+\|\nabla f\|^2_\infty \alpha^2 \|(u^+)^{2\alpha-2}\|_{B,p_*} +\frac{\alpha^2}{2\alpha-1}\|\nabla\eta\|_\infty \|\nabla f\|_\infty  \|(u^+)^{2\alpha-1}\|_{B,p_*}\Bigr].
\end{multline*}
By Jensen's inequality it holds
\[
 \|u^+\|_{B,(2\alpha-2)p_*}\leq  \|u^+\|_{B,2\alpha p_*},\quad\|u^+\|_{B,(2\alpha-1)p_*}\leq  \|u^+\|_{B,2\alpha p_*},
\]
therefore we can rewrite and get the desired result
\begin{multline*}
\|\eta u^+\|_{B,\alpha\rho}^{2\alpha} \leq 2 C_1\|\lambda^{-1}\|_{B,q}\|\Lambda\|_{B,p}|B|^\frac{2}{d}\Bigl[(\alpha^2+1)\| u^+\|^{2\alpha}_{B,2\alpha p_*}\|\nabla\eta\|^2_\infty\\+\|\nabla f\|^2_\infty \alpha^2 \|u^+\|^{2\alpha-2}_{B,2\alpha p_*} +\frac{\alpha^2}{2\alpha-1}\|\nabla\eta\|_\infty \|\nabla f\|_\infty  \|u^+\|^{2\alpha-1}_{B,2\alpha p_*}\Bigr].
\end{multline*}
Finally, absorbing the mixed product in the two squares we obtain \eqref{eq:meanvalue}.
\end{proof}
Clearly the same result holds, with the same constant, also for supersolutions with $u^+$ replaced by $u^-$. It is then clear that we can get the same type of inequality for solutions to \eqref{eq:Poisson}. This is the content of the next corollary.

\begin{corollary}\label{cor} Let $u\in \F_{loc}$ be a solution of \eqref{eq:Poisson} in $B$. Let $\eta\in C_0^\infty(B)$ be a cut-off function. Then there exists a constant $C_1>0$ such that for all $\alpha\geq 1$
\begin{equation}\label{eq:corollary}
\|\eta u\|_{B,\alpha\rho}^{2\alpha} \leq \alpha^2  C_1\|\lambda^{-1}\|_{B,q}\|\Lambda\|_{B,p}|B|^\frac{2}{d}\Bigl[\|\nabla\eta\|^2_\infty\| u\|^{2\alpha}_{B,2\alpha p_*}+\|\nabla f\|^2_\infty \|u\|^{2\alpha-2}_{B,2\alpha p_*}\Bigr].
\end{equation}
\end{corollary}
\begin{proof} The proof is trivial, since $u$ is both a subsolution and a supersolution of \eqref{eq:Poisson}. Moreover, $u=u^+-u^-$ and $\|u^+\|_r\vee \|u^-\|_r\leq\|u\|_r$.
\end{proof}

\begin{theorem}\label{thm:meanvalueinequality} Fix a point $x_0\in \R^d$ and $R>0$. Denote by $B(R)$ the ball of center $x_0$ and radius $R$. Suppose that $u$ is a solution in $B(R)$ of \eqref{eq:Poisson}, and assume that $|\nabla f|\leq c_f/R$. Then for any $p,q\in (1,\infty]$ such that $1/p+1/q<2/d$, $d\geq 2$, there exist $\kappa\defeq \kappa(q,p,d)\in (1,\infty)$, $\gamma\defeq \gamma(q,p,d)\in (0,1]$ and $C_2\defeq C_2(q,p,d,c_f)>0$ such that
\begin{equation}\label{eq:moser2}
\|u\|_{B(\sigma'R),\infty}\leq C_2
\biggl(\frac{1\vee\|\lambda^{-1}\|_{B(R),q}\|\Lambda\|_{ B(R),p}}{(\sigma-\sigma')^2 }\biggr)^\kappa \|u\|_{B(\sigma R),\rho}^\gamma\vee\|u\|_{B(\sigma R),\rho},
\end{equation}
for any fixed $1/2\leq \sigma'<\sigma\leq 1$.
\end{theorem}
\begin{proof} We are going to apply inequality  \eqref{eq:corollary} iteratively. For fixed $1/2\leq \sigma'<\sigma\leq 1$, and $k\in\N$ define
\[
\sigma_k=\sigma'+2^{-k+1}(\sigma-\sigma').
\]
It is immediate that $\sigma_k-\sigma_{k+1}=2^{-k+1}(\sigma-\sigma')$ and that $\sigma_1 = \sigma$, furthermore $\sigma_k\downarrow \sigma'$. We have already observed that $\rho>2p^*$, where $p^*$ is the H\"older's conjugate of $p$. Set $\alpha_k \defeq (\rho/2 p^*)^k$, $k\geq1$, clearly $\alpha_k>1$ for all $k\geq 1$. Finally consider a cutoff $\eta_k$ which is identically $1$ on $B(\sigma_{k+1} R)$ and $\eta_k = 0$ on $\partial B(\sigma_k R)$, assume that $\eta_k$ has a linear decay on $B(\sigma_k R)\setminus B(\sigma_{k+1} R)$, i.e. chose $\eta_k$ in such a way that $\|\nabla \eta_k\|_\infty \leq 2^{k} / (\sigma-\sigma')R$.

An application of Corollary \ref{cor} and of the relation $\alpha_k\rho = 2\alpha_{k+1}p^*$, yields
\begin{align*}
\|u&\|_{B(\sigma_{k+1}R),2\alpha_{k+1}p^*}\\
&\leq  \biggl(C\frac{2^{2k}\alpha_k^2|B(\sigma_k R)|^\frac{2}{d}}{(\sigma-\sigma')^2 R^2}\|\lambda^{-1}\|_{B(\sigma_k R),q}\|\Lambda\|_{ B(\sigma_k R),p}\biggr)^\frac{1}{2\alpha_k} \|u\|_{B(\sigma_k R),2\alpha_k p_*}^{\gamma_k}\\
&\leq
\biggl(C\frac{2^{2k}\alpha_k^2}{(\sigma-\sigma')^2 }\|\lambda^{-1}\|_{B(R),q}\|\Lambda\|_{ B(R),p}\biggr)^\frac{1}{2\alpha_k} \|u\|_{B(\sigma_k R),2\alpha_k p^*}^{\gamma_k},
\end{align*}
where $\gamma_k = 1$ if $\|u\|_{B(\sigma_k R),2\alpha_k p^*}\geq1$ and $\gamma_k = 1-1/\alpha_k$ otherwise. We can iterate the inequality above and stop at $k=1$, so that we get
\begin{equation*}
\|u\|_{B(\sigma_{j+1}R),2\alpha_{j+1}p^*}\leq
\prod_{k=1}^j\biggl(C\frac{(\rho/p^*)^{2k}}{(\sigma-\sigma')^2 }\|\lambda^{-1}\|_{B(R),q}\|\Lambda\|_{ B(R),p}\biggr)^\frac{1}{2\alpha_k} \|u\|_{B(\sigma R),\rho}^{\prod_{k=1}^j\gamma_k}.
\end{equation*}
Observe that $\kappa\defeq\frac{1}{2}\sum1/\alpha_k<\infty$, $\sum k/\alpha_k<\infty$ and that
\[
\|u\|_{B(\sigma'R),2\alpha_j p^*}\leq \biggl(\frac{|B(\sigma_kR)|}{|B(\sigma'R)|}\biggr)^\frac{1}{2\alpha_j p^*}\|u\|_{B(\sigma_jR),2\alpha_j p^*}\leq K\|u\|_{B(\sigma_jR),2\alpha_j p^*},
\]
for some $K$ and all $j\geq 1$. Hence, taking the limit as $j\to\infty$, gives the inequality
\begin{equation*}
\|u\|_{B(\sigma'R),\infty}\leq C_2
\biggl(\frac{1\vee\|\lambda^{-1}\|_{B(R),q}\|\Lambda\|_{ B(R),p}}{(\sigma-\sigma')^2 }\biggr)^\kappa \|u\|_{B(\sigma R),\rho}^{\prod_{k=1}^\infty \gamma_k}.
\end{equation*}
Define $\gamma\defeq \prod_{k=1}^\infty (1-1/\alpha_k) \in (0,1]$. Then, $0<\gamma \leq \prod_{k=1}^\infty \gamma_k\leq 1$ and the above inequality can be written as
\[
\|u\|_{B(\sigma'R),\infty}\leq C_2
\biggl(\frac{1\vee\|\lambda^{-1}\|_{B(R),q}\|\Lambda\|_{ B(R),p}}{(\sigma-\sigma')^2 }\biggr)^\kappa \|u\|_{B(\sigma R),\rho}^\gamma\vee\|u\|_{B(\sigma R),\rho}.
\]
which is the desired inequality.
\end{proof}

The previous inequality can be improved. This is what the next Corollary is about. For the proof we follow the argument of \cite{saloff2002aspects}[Theorem 2.2.3].
\begin{corollary} Suppose that $u$ satisfies the assumptions of Theorem \ref{thm:meanvalueinequality}. Then, for all $\alpha\in(0, \infty)$ and for any $1/2\leq \sigma'<\sigma < 1$ there exist $C_3:=C_3(q,p,d,c_f)>0$, $\gamma'\defeq \gamma'(\gamma,\alpha,\rho)$ and $\kappa'\defeq \kappa'(\kappa,\alpha,\rho)$, such that
\begin{equation}\label{cor:moser}
\|u\|_{B(\sigma'R),\infty}\leq C_3
\biggl(\frac{1\vee\|\lambda^{-1}\|_{B(R),q}\|\Lambda\|_{ B(R),p}}{(\sigma-\sigma')^2 }\biggr)^{\kappa'} \|u\|_{B(\sigma R),\alpha}^{\gamma'}\vee\|u\|_{B(\sigma R),\alpha}.
\end{equation}
\end{corollary}
\begin{proof} From inequality \eqref{eq:moser2} we get
\begin{equation*}
\|u\|_{B(\sigma'R),\infty}\leq C_2
\biggl(\frac{1\vee\|\lambda^{-1}\|_{B(R),q}\|\Lambda\|_{ B(R),p}}{(\sigma-\sigma')^2 }\biggr)^\kappa \|u\|_{B(\sigma R),\rho}^{\gamma}\vee \|u\|_{B(\sigma R),\rho}.
\end{equation*}
Hence, the result follows immediately for $\alpha>\rho$ by means of Jensen's inequality. For $\alpha\in(0,\rho)$ we use again an iteration argument. Consider $\sigma_k=\sigma-2^{-k}(\sigma-\sigma')$. By H\"older's inequality we get
\[
\|u\|_{B(\sigma_k R),\rho}\leq \|u\|_{B(\sigma_k R),\alpha}^\theta\|u\|_{B(\sigma_k R),\infty}^{1-\theta}
\]
with $\theta= \alpha/\rho$. An  application of inequality \eqref{eq:moser2} gives
\[
\|u\|_{B(\sigma_{k-1}R),\infty}\leq\\ 2^{2\kappa k} J \|u\|_{B(\sigma R),\alpha}^{\gamma_k\theta}\|u\|_{B(\sigma_k R),\infty}^{\gamma_k-\gamma_k\theta},
\]
here $\gamma_k=1$ if $\|u\|_{B(\sigma_k R),\rho}\geq 1$, $\gamma_k = \gamma$ otherwise and $J = c(1\vee\|\lambda^{-1}\|_{B(R),q}\|\Lambda\|_{ B(R),p}/(\sigma-\sigma')^2 )^{\kappa}$, where $c$ is a constant that can be taken greater than one.

By iteration from $k=1$ up to $i>1$, via similar computations as the Theorem \ref{thm:meanvalueinequality}, we get
\begin{equation*}
\|u\|_{B(\sigma' R),\infty}\leq (J 2^{2\kappa})^{\sum_{k=1}^{i}k(1-\theta)^{k-1}}\Bigl(\|u\|_{B(\sigma R),\alpha}^{\gamma\theta\sum_{k=1}^{i}(\gamma-\gamma\theta)^{k-1}}\vee \|u\|_{B(\sigma R),\alpha}^{\theta\sum_{k=1}^{i}(1-\theta)^{k-1}}\Bigr) \|u\|_{B(\sigma R),\infty}^{\beta_i}
\end{equation*}
where $\beta_i\to 0$ as $i\to \infty$.
which gives the desired result taking the limit as $i\to\infty$. In particular we get $\gamma' = \gamma \theta / (1-\gamma+\gamma \theta)$.
\end{proof}

\subsection{Existence of the Minimal Diffusion}\label{sec:mindiff}

In the context of diffusions in random environment we would like to be able to fix a common starting position for almost all realizations of the environment, or alternatively to start the process from all possible positions $x\in\R^d$. To achieve this aim we assume the following:
\begin{itemize}
\item[$(b.3)$\namedlabel{ass:b.3}{$(b.3)$}] $\lambda^{-1}(x),\Lambda(x)\in L^\infty_{loc}(\R^d)$.
\end{itemize}
Recall that the resolvent $G_\alpha^{B,\theta}$ restricted to $B$ of a diffusion process $\mathbf{M}^{\theta}\defeq(X_t^\theta,\PR^\theta_x,\zeta^\theta)$ is defined by
\[
G_\alpha^{B,\theta} f(x)\defeq\mean^\theta_x\biggl[\int_0^{\tau_B} e^{-\alpha t}f(X^\theta_t)\,dt\biggr],\quad f\geq0
\]
being $\tau_B = \inf\{t>0: X^\theta_t\in B^c\}$.  When $\theta\equiv 1$ we will drop it from the notation.
\begin{theorem}\label{thm:min_diff}
Assume \ref{ass:b.1}, \ref{ass:b.2}, \ref{ass:b.3}, and  $\theta,\theta^{-1}\in L^\infty_{loc}(\R^d)$. Denote by $C_\infty(B)$ the set of continuous functions vanishing at the boundary. Then, there exists a unique standard diffusion process $\mathbf{M}^{\theta}\defeq(X_t^\theta,\PR^\theta_x,\zeta^\theta)$, $x\in \R^d$ whose resolvent $G^{B,\theta}_\alpha$ restricted to any open bounded set $B$ satisfies
\[
G_\alpha^{B,\theta} f \in C_\infty(B),\quad f\in L^p(B,\theta dx),\quad p>d
\]
and $G_\alpha^{B,\theta} C_\infty(B)$ is dense in $C_\infty(B)$.
\end{theorem}

\begin{proof}
For a proof see for example  \cite{ichihara}, \cite{kunita1970}, \cite{tomisaki}.
\end{proof}

We will consider from now on only the process $\mathbf{M}^\theta$ constructed in Theorem \ref{thm:min_diff}.
Fix a ball $B\subset \R^d$ and consider the semigroup associated to the process above killed when exiting from $B$, then its semigroup is given by
\[
\mathcal{P}_t^{B,\theta} f(x)\defeq \mean_x[f(X^{\theta}_t),t<\tau_B],
\]
By Theorem \ref{thm:min_diff} and Hille-Yoshida's Theorem, $\mathcal{P}_t^{B,\theta} C_\infty(B)\subset C_\infty(B)$. Such a property turns out to be very handy to remove all the ambiguities about exceptional sets and to construct a transition kernel $p_t^{B,\theta}(x,y)$ for $\mathcal{P}_t^{B,\theta}$ which is jointly continuous in $x,y$. This is the content of the next theorem whose proof is a slight variation of \cite[Theorem 2.1]{Barlow99non-localdirichlet} since we assume to have a Feller semigroup.

\begin{theorem} \label{jointcont} Let $B\subset \R^d$ a ball and $\mathcal{P}_t$ be a Feller semigroup on $L^2(B,m)$, i.e. $\mathcal{P}_t C_\infty(B)\subset C_\infty(B)$. Assume that
\begin{equation}\label{ass:ultracontractivity}
\|\mathcal{P}_t f\|_\infty\leq M(t)\|f\|_1,
\end{equation}
for all $f\in L^1(B,m)$ and $t>0$ and some lower semicontinuous function $M(t)$ on $(0,\infty)$. Then there exists a positive symmetric kernel $p_t(x,y)$ defined on $(0,\infty)\times B\times B$ such that
\begin{itemize}
\item[(i)] $\mathcal{P}_t(x,dy) = p_t(x,y) m(dy)$, for all $x\in B$, $t>0$,
\item[(ii)] for every $t,s>0$ and $x,y\in B$
\[
p_{t+s}(x,y)=\int_B p_t(x,z)p_s(z,y)m(dz),
\]
\item[(iii)] $p_t(x,y)\leq M(t)$ for every $t>0$ and $x,y\in B$,
\item[(iv)] for every fixed $t>0$, $p_t(x,y)$ is jointly continuous in $x,y\in B$.
\end{itemize}
\end{theorem}

We see that if we choose $m(dx)= \theta(x) dx$ and we assume \ref{ass:b.1}, \ref{ass:b.2}, \ref{ass:b.3} we immediately get the existence of a transition kernel $p^{B,\theta}_t(x,y)$ for the semigroup $\mathcal{P}_t^{B,\theta}$, jointly continuous in $x,y\in B$. Indeed assumption \eqref{ass:ultracontractivity} is easily satisfied by \ref{ass:b.3}. In the next proposition we prove the existence of a transition kernel $p_t^\theta(x,y)$ for the semigroup $\mathcal{P}^\theta_t$ of $\mathbf{M}^\theta$ by a localization argument.

\begin{proposition} Assume \ref{ass:b.3} and $\theta,\theta^{-1}\in L^\infty_{loc}(\R^d)$. Consider the semigroup $\mathcal{P}_t^\theta$ associated to the minimal diffusion $\mathbf{M}^\theta$. Then, there exists a transition kernel $p_t^{\theta}(x,y)$ defined on $(0,\infty)\times \R^d\times \R^d$ associated to $\mathcal{P}_t^\theta$,
\[
\mathcal{P}_t^\theta f(x)=\int_{\R^d}f(y)p_t^{\theta}(x,y)\theta(y)\,dy,\quad \forall x\in\R^d,\,t>0.
\]
Moreover, for all $t>0$ and $x,y\in \R^d$
\[
 p^{B_R,\theta}_t(x,y) \nearrow  p^\theta_t(x,y),\quad R\to\infty,
\]
being the limit increasing in $R$.
\end{proposition}
\begin{proof} The proof comes from the the fact that for all balls $B\subset\R^d$ the semigroup $\mathcal{P}_t^{B,\theta}$ satisfies \eqref{ass:ultracontractivity}, which means that $\mathcal{P}_t^{\theta}$ is locally ultracontractive and from Theorem 2.12 of \cite{grigor'yan2012}.
\end{proof}

As a further consequence of assumption \ref{ass:b.3}, more precisely from the fact that $\lambda$ is locally bounded from below we can prove that $\mathbf{M}^\theta$ is an irreducible process.

\begin{proposition}\label{prop:irr} Assume \ref{ass:b.3} and assume $\theta^{-1},\theta\in L^\infty_{loc}(\R^d)$. Then the process $\mathbf{M}^\theta$ is irreducible.
\end{proposition}
\begin{proof} It follows immediately from Corollary 4.6.4. in \cite{fukushima1994dirichlet}.
\end{proof}

In the next theorem we clarify the relation between $\mathbf{M}$ and $\mathbf{M}^\theta$, namely, we show that $\mathbf{M}^\theta$ can be obtained by $\mathbf{M}$ through a time change.

\begin{theorem}[Time change]\label{thm:time_change} Assume \ref{ass:b.3} and assume $\theta^{-1},\theta\in L^\infty_{loc}(\R^d)$. Define $\hat{\mathbf{M}}=(\hat{X}_t,\PR_x)$ by
\[
\hat{X}_t\defeq X_{\tau_t},\quad\tau_t=\inf\{s>0;\int_0^s\theta(X_u)\,du>t\},
\]
then $\hat{\mathcal{P}}_tf(x)=\mean_x[f(X_{\tau_t})]=\mathcal{P}_t^\theta f(x)$ for almost all $x\in\R^d$, $t>0$ and $f:\R^d\to\R$ positive and measurable.
\end{theorem}
\begin{proof} According to Theorem 6.2.1 of \cite{fukushima1994dirichlet},  $\hat{\mathcal{P}}_tf(x)=\mathcal{P}_t^\theta f(x)$ coincide for almost all $x\in\R^d$ and $t>0$.
\end{proof}

There is a natural time change $\theta:\R^d\to\R_{\geq0}$ which makes the process $\mathbf{M}^\theta$ conservative. Namely we pick $\theta\equiv \Lambda$. The condition we give will be suitable in the setting of Ergodic environment, and in particular, is a consequence of \ref{ass:b.2}.

\begin{proposition}\label{prop:explosion} Assume that
\[
\limsup_{R\to\infty} \frac{1}{|B_R|}\int_{B_R} \Lambda(x)\,dx<\infty.
\]
Then the process $\mathbf{M}^\Lambda$ is conservative.
\end{proposition}
\begin{proof} The proof is an application of Theorem 5.7.3 of \cite{fukushima1994dirichlet}.
\end{proof}

\section{Diffusions in Random Environment}\label{sec:diffre}

\subsection{Construction of the Process in Random Environment}
By a stationary and ergodic random environment $(\Omega,\mathcal{G},\mu,\{\tau_x\}_{\R^d})$, we mean a probability space $(\Omega,\mathcal{G},\mu)$ on which is defined a group of transformations $\{\tau_x\}_{x\in\R^d}$ acting on $\Omega$ such that
\begin{itemize}
\item[(i)] $\mu(\tau_x A) = \mu(A)$ for all $A\in\mathcal{G}$ and any $x\in\R^d$;
\item[(ii)] if $\tau_x A = A$ for all $x\in\R^d$, then $\mu(A)\in\{0,1\}$;
\item[(iii)] the function $(x,\omega)\to \tau_x\omega$ is $\mathcal{B}(\R^d)\otimes\mathcal{G}$-measurable.
\end{itemize}
Let us consider the following bilinear form
\begin{equation*}
\Ew(u,v)\defeq \sum_{i,j}\int_{\R^d} a^\omega_{ij}(x)\partial_i u(x)\partial_j v(x) dx,\quad u,v\in C_0^\infty(\R^d),
\end{equation*}
where $a^\omega_{ij}(x)$ satisfies \ref{ass:a.1}, \ref{ass:a.2} and \ref{ass:a.3} of Section 1.

Throughout this section we will look at two Dirichlet forms determined by $\E^\omega$ above. One is the Dirichlet form $(\E^\omega,\F^{\omega})$ on $L^2(\R^d,dx)$ where $\F^\omega$ is the completion of $C_0^\infty(\R^d)$ in $L^2(\R^d,dx)$ with respect to $\E^\omega_1 := \E^\omega+(\cdot,\cdot)$. The second is the Dirichlet form $(\E^\omega,\F^{\Lambda,\omega})$ on $L^2(\R^d,\Lambda^\omega dx)$ where $\F^{\Lambda,\omega}$ is the completion of $C_0^\infty(\R^d)$ in $L^2(\R^d,\Lambda^\omega dx)$ with respect to $\E^\omega_1 := \E^\omega+(\cdot,\cdot)_\Lambda$.

We have already observed that \ref{ass:a.1}, \ref{ass:a.2} and \ref{ass:a.3} imply \ref{ass:b.1}, \ref{ass:b.2} and \ref{ass:b.3} of Section \ref{sec:sobmos}, for $\mu$-almost all $\omega\in\Omega$.
In particular, by Theorem \ref{thm:min_diff}, we have  the existence, for $\mu$-almost all $\omega\in\Omega$, of two minimal diffusion processes,
$\mathbf{M}^{\omega} = (X_t^{\omega},\PR_x^{\omega},\zeta^\omega)$ and $\mathbf{M}^{\Lambda,\omega} = (X_t^{\Lambda,\omega},\PR_x^{\Lambda,\omega})$, respectively associated to $(\E^\omega,\F^{\omega})$ and $(\E^\omega,\F^{\Lambda,\omega})$. Denote by $\mathcal{P}_t^\omega$ the semigroup associated to $\mathbf{M}^\omega$ and by $p_t^\omega(x,y)$ its transition kernel with respect to $dx$. Analogously, denote by $\mathcal{Q}_t^\omega$ the semigroup associated to $\mathbf{M}^{\Lambda,\omega}$ and by $q_t^\omega(x,y)$ its transition kernel with respect to $\Lambda^\omega(x)dx$.

\begin{lemma}[Translation Property for killed process] Fix a ball $B\subset\R^d$. Then for $\mu$-almost all $\omega\in\Omega$
\begin{align}\label{eq:homogeneity}
p_t^{B-z,\tau_z\omega}(x-z,y-z)&=p_t^{B,\omega}(x,y),\\q_t^{B-z,\tau_z\omega}(x-z,y-z)&=q_t^{B,\omega}(x,y),\notag
\end{align}
for all $t\geq0$, $x,y\in B$ and $z\in\R^d$.
\end{lemma}
\begin{proof}
We prove property \eqref{eq:homogeneity} only for the semigroup $\mathcal{Q}^{B,\omega}_t$, being the other equivalent. It is known in \cite{fukushima1994dirichlet} that the resolvent $G_\alpha^{B,\omega}$ is uniquely determined by the following equation
\[
\E^\omega_\alpha(G^{B,\omega}_\alpha f,v)=\int_B f(x)v(x) \Lambda(x;\omega)\,dx,
\]
for all $f\in L^2(B)$, $v\in W^2_0(B)$.
On the other hand
\begin{align*}
\E^\omega_\alpha(G^{B,\omega}_\alpha f,v) &=\int_{B-z} f(x+z)v(x+z) \Lambda(x;\tau_z\omega)\,dx\\
	&=\E^{\tau_z\omega}_\alpha([G^{B-z,\tau_z\omega}_\alpha f(\cdot+z)],v(\cdot+z))\\
	&= \E^\omega_\alpha([G^{B-z,\tau_z\omega}_\alpha f(\cdot+z)](\cdot-z),v),
\end{align*}
for all $f\in L^2(B)$, $v\in W^2_0(B)$. Hence, for $\mu$-almost all $\omega\in\Omega$
\[
[G^{B-z\tau_z\omega}_\alpha f(\cdot+z)](x-z)=G^{B,\omega}_\alpha f(x),\quad \mbox{a.a }x\in B, \forall z\in\R^d.
\]
Moving from the resolvent to the semigroup we get the relation
\[
[\mathcal{Q}^{B-z,\tau_z\omega}_t f(\cdot+z)](x-z)=\mathcal{Q}^{B,\omega}_t f(x),
\]
for all $f\in C_\infty(B)$. The equality is true for all $x\in B$ and for all $z\in\R^d$ by the Feller property, $\mu$-almost surely.
Finally it is easy to derive the equality for the transition kernel and get
\begin{equation}\label{eq:hom}
q_t^{B-z,\tau_z\omega}(x-z,y-z)=q_t^{B,\omega}(x,y),
\end{equation}
for all $z\in\R^d$, and almost all $x,y\in B$, $\mu$-almost surely. Using the joint continuity of $q_t^{B,\omega}(x,y)$ in $x$ and $y$ (cf. $(iv)$ Theorem  \ref{jointcont}) we get \eqref{eq:hom} for all $z\in\R^d$, $x,y\in B$, $\mu$-almost surely.
\end{proof}

\begin{lemma}[Translation Property] For $\mu$-almost all $\omega\in\Omega$
\begin{align}\label{eq:homogeneity2}
p_t^{\tau_z\omega}(x-z,y-z)&=p_t^{\omega}(x,y),\\q_t^{\tau_z\omega}(x-z,y-z)&=q_t^{\omega}(x,y),\notag
\end{align}
for all $t\geq 0$ and $x,y,z\in\R^d$
\end{lemma}
\begin{proof}
It follows from the previous lemma, passing to the limit. Namely, take an increasing sequence of balls $B_n\uparrow \R^d$, then we have
\begin{align*}
p_t^{\tau_z\omega}(x-z,y-z) & = \lim_{n\to\infty}p_t^{B_n-z,\tau_z\omega}(x-z,y-z)\\& =\lim_{n\to\infty}p_t^{B_n,\omega}(x,y)=p_t^{\omega}(x,y).
\end{align*}
\end{proof}

\subsection{Environment Process}

We shall first construct the environment process for $\mathbf{M}^{\Lambda.\omega}=(X^{\Lambda,\omega}_t,\PR^{\Lambda,\omega}_x)=:(Y^{\omega}_t,\mathbb{Q}^{\omega}_x)$, $x\in\R^d$, since we know that it is conservative $\mu$-almost surely by Proposition \ref{prop:explosion}. From this construction and the Ergodic theorem we will prove that also the process $\mathbf{M}^\omega$ is conservative $\mu$-almost surely.

For a fixed $\omega\in\Omega$, we define a stochastic process on $\Omega$ by
\[
\eta_t^\omega(\tilde{\omega})\defeq \tau_{Y_t^\omega(\tilde{\omega})}\omega,\quad t\geq 0
\]
where $\tilde{\omega}$ is a point of the sample space of the diffusion $\mathbf{M}^{\Lambda,\omega}$. The process $\eta_t^\omega$ under the measure $\mathbb{Q}^{\omega}_x$ is $\Omega$ valued and it is known as the environment process.
First, we describe the semigroup associated to $\eta_t^\omega$ under $\mathbb{Q}^{\omega}_0$. Take any positive and bounded $\mathcal{G}$-measurable function $f:\Omega\to \R$ and observe that
\[
\mathbf{Q}_t f(\omega)\defeq\mean_0^{\omega} [f(\tau_{Y^\omega_t}\omega)] = \mathcal{Q}_t^\omega f(\tau_.\omega)(0)=\int_{\R^d}f(\tau_y\omega) q^\omega_t(0,y)\Lambda(\tau_y \omega)\,dy.
\]
\begin{proposition}\label{prop:symmetric}
$\{\mathbf{Q}_t\}_{t\geq 0}$ defines a symmetric strongly continuous semigroup on $L^2(\Omega,\Lambda d\mu)$, the process $t\to\eta_t^\omega$ is ergodic with respect to $\mu$.
\end{proposition}
\begin{proof} The proof of the contractivity, the symmetry and the strong continuity of $\{\mathbf{Q}_t\}_{t\geq 0}$ on $L^2(\Omega,\Lambda d\mu)$ follows from the stationarity of the environment and by \eqref{eq:homogeneity2}, it is standard and can be found in \cite{Osada1983}, \cite{zhikov1994homogenization}.

The proof ot the ergodicity of the process $t\to\eta_t^\omega$ with respect to $\Lambda d\mu$ can also be found in \cite{Osada1983} and it is based on the irreducibility of the process $Y^\omega_t$, which was proven in Proposition \ref{prop:irr}.
\end{proof}\

\begin{proposition}[Ergodic Theorem]\label{prop:ergodic}
For all functions $f\in L^p(\Omega,\Lambda d\mu)$, $p\geq 1$, set $f(x;\omega)=f(\tau_x\omega)$, then
\[
\lim_{t\to\infty}\frac{1}{t}\int_0^t f(Y^\omega_s;\omega)\,ds = \mean_\mu[f\Lambda],\quad \mathbb{Q}_x^\omega\mbox{-a.s}, \mbox{ a.a. }x\in\R^d,
\]
for $\mu$-almost all $\omega\in\Omega$.
\end{proposition}
\begin{proof}  In order to have the result stated, observe that the measure $Q^{\tau_x\omega}_0$ induced by $\mathbb{Q}_0^{\tau_x\omega}$ through $\eta_t^{\tau_x\omega}$ on the space of $\Omega$-valued trajectories coincides with the measure $Q^{\omega}_x$ induced by $\mathbb{Q}_x^\omega$ through $\eta^{\omega}_t$ . It is then easy to show that for any ball $B\subset\R^d$ the two measures
\[
\int_{\Omega}Q_0^\omega(\cdot) d\mu=\frac{1}{|B|}\int_{B\times \Omega}Q_0^{\tau_x\omega}(\cdot) dx d\mu = \frac{1}{|B|}\int_{\Omega\times B}Q_x^\omega(\cdot) d\mu dx
\]
coincide; in the first equality we used the stationarity of the environment. The fact that the limiting relation hold $\int Q_0^\omega(\cdot) d\mu$-almost surely follows immediately from Proposition \ref{prop:symmetric}, then the result follows.
\end{proof}
We use Proposition \ref{prop:ergodic} to control the explosion time of the process $\mathbf{M^\omega}=(X_t^\omega,\PR_x^\omega,\zeta^\omega)$ in terms of the time changed process $\mathbf{M}^{\Lambda,\omega}$. Indeed consider the time change
\[
\tau_t\defeq\inf\Bigl\{s>0:\int_{0}^{s}\frac{1}{\Lambda(Y_u^\omega,\omega)}\,du>t\Bigr\},
\]
and define the process $\hat{Y}^\omega_t=Y_{\tau_t}^\omega$. We know, by Theorem \ref{thm:time_change} that $\hat{Y}^\omega_t$ is a version of $X^\omega_t$. It is not difficult to see that the explosion time of $\hat{Y}^\omega_t$ equals $\int_{0}^{\infty}\frac{1}{\Lambda(Y_u^\omega,\omega)}\,du$ \cite[see chapter 6]{fukushima1994dirichlet}.
By Proposition \ref{prop:ergodic},
\[
\lim_{t\to\infty}\frac{1}{t}\int_0^t \frac{1}{\Lambda(Y^\omega_s;\omega)}\,ds = \mean_\mu[\Lambda^{-1}\Lambda] = 1,\quad\mathbb{Q}_x^\omega\mbox{-a.s}, \mbox{ a.a. }x\in\R^d,
\]
for $\mu$-almost all $\omega\in\Omega$.
It follows that $\hat{Y}^\omega_t$ is conservative for almost all starting points $x\in\R^d$, $\mu$-almost surely. This, together with Theorem \ref{thm:time_change} leads to the following result.
%
\begin{theorem} Let $\mathbf{M}^\omega=(X_t^\omega, \PR_x^\omega,\zeta^\omega)$, $x\in\R^d$, be the minimal diffusion constructed in section 3.1. Then such a diffusion is conservative.
\end{theorem}
\begin{proof} By Theorem \ref{thm:time_change}, $\mathcal{P}_t^\omega 1(x) = \hat{\mathcal{P}}_t^\omega 1(x) = 1$ for almost all $x\in\R^d$, and since $\mathbf{M}^\omega$ is our minimal diffusion, then $\mathcal{P}_t^\omega 1(x)=1$ for all $x\in\R^d$. We can pass from almost all to all $x\in\R^d$ since the minimal diffusion satisfies property (4.2.9) in \cite{fukushima1994dirichlet}, namely $\mathcal{P}_t^\omega(x,dy)$ is absolutely continuous with respect to the Lebesgue measure for each $t>0$ and each $x\in\R^d$ (see Theorem 4.5.4 in \cite{fukushima1994dirichlet}).
\end{proof}

From now on we will completely forget about the time changed process. Following the construction in this section it is possible to obtain an environment process for the minimal diffusion $\mathbf{M}^\omega=(X_t^\omega, \PR_x^\omega)$, namely the process $t\to\tau_{X^\omega_t}\omega=:\psi_t^\omega$, with semigroup $\mathbf{P}_t$, which is precisely given by
\[
\mathbf{P}_t f(\omega)\defeq\int_{\R^d}f(\tau_y\omega) p^\omega_t(0,y)\,dy.
\]

\begin{proposition}
$\{\mathbf{P}_t\}_{t\geq 0}$ defines a symmetric strongly continuous semigroup on $L^2(\Omega,d\mu)$, and $t\to\psi_t^\omega$ is ergodic with respect to $\mu$.
\end{proposition}
\begin{proof} Analogous to Proposition \ref{prop:symmetric}.
\end{proof}

\section{Corrector and Harmonic coordinates}\label{sec:corr}
\label{sec:corrector}

\subsection{Space $L^2(a)$ and Weyl's decomposition.}

Fix a stationary and ergodic random medium $(\Omega,\mathcal{G},\mu,\tau_x)$. In this section we rely only on assumption \ref{ass:a.1} and $\mean_\mu[\lambda^{-1}]$, $\mean_\mu[\Lambda]$ finite.

In order to construct the corrector, we introduce the following space
\[
L^2(a)\defeq \bigl\{ V:\Omega\to \R^d : \mean_\mu[\langle a V,V\rangle]<\infty\bigr\}.
\]
Such a space is clearly a pre-Hilbert space with the scalar product
\[
\Theta(U,V)\defeq\mean_{\mu}[\langle aU,V\rangle].
\]
$L^2(a)$ is isometric to $L^2(\Omega,\mu)^d$ through the map $\Psi: L^2(\Omega,\mu)^d\to L^2(a)$ given by $\Psi(V)=a^{-1/2}V$. In particular $L^2(a)$ is an Hilbert space. Notice that as a consequence of \ref{ass:a.1}, $\mean_\mu[\lambda^{-1}],\,\mean_\mu[\Lambda]<\infty$ and H\"older's inequality we have that $L^2(a)\subset L^1(\Omega,\mu)$.

The group $\{\tau_x\}_{\R^d}$ on $\Omega$ defines a group of strongly continuous unitary operators $\{T_x\}_{\R^d}$ on $L^r(\Omega,\mu)$ for any $r>1$, by the position $T_x (V)=V\circ \tau_x$, see \cite[Chapter 7]{zhikov1994homogenization}. Therefore, $\{T_x\}_{x\in\R^d}$ on $L^2(\Omega,\mu)$ defines the closed operators $D_i$ for $i=1,...,d$, by
\[
D_i U \defeq \lim_{h\to 0}\frac{T_{he_i}U-U}{h},
\]
where the limit is taken in $L^2(\Omega,\mu)$. Denote by $\mathcal{D}(D_i)$ the domain of $D_i$. We shall consider the following class of smooth functions
\begin{equation}\label{eq:setC}
\mathcal{C}\defeq\Bigl\{\int_{\R^d}f(\tau_x\omega) \phi(x)dx\, |\, f\in L^\infty(\Omega), \phi\in C_0^\infty(\R^d)\Bigr\}.
\end{equation}
It can be proved that if $v\in\mathcal{C}$,
\[
v(\omega)=\int_{\R^d}f(\tau_x\omega) \phi(x)dx \Rightarrow D_i v(\omega)=-\int_{\R^d}f(\tau_x\omega)\partial_i\phi(x)dx.
\]
In particular, $v\in\bigcap_{i=1}^d\mathcal{D}(D_i)$.
It is also clear that $\nabla v=(D_1v,\dots,D_dv)\in L^2(a)$ and that $x\to v(\tau_x\omega)\in C^\infty(\R^d)$ for $\mu$-almost all $\omega\in\Omega$. We define the space of potential $L^2_{pot}$ to be the closure of $\{\nabla v | v\in\mathcal{C}\}$ in $L^2(a)$.

\begin{lemma} \label{lemma:L2pot} Let $U\in L^2_{pot}$. Then $U$ satisfies the
following properties
\begin{itemize}
 \item[(i)] $\mean_\mu[U_i]=0$ for all $i=1,\dots,d$.
 \item[(ii)] for all $\eta\in C^\infty_0(\R^d)$ and $i,j=1,\dots,d$
 \[
  \int_{\R^d} U_i(\tau_x\omega) \partial_j\eta(x)\,dx=\int_{\R^d}
U_j(\tau_x\omega) \partial_i\eta(x)\,dx,
 \]
 for $\mu$-almost all $\omega\in\Omega$.
\end{itemize}
\end{lemma}

\begin{proof} In both cases the proof follows simply by considering functions
of the type $\nabla f$ such that $f\in\mathcal{C}$. Then conclude by density.

Let start with $(i)$. Observe that if $f\in\mathcal{C}$ then
\[
\mean_\mu[D_i f] = \lim_{h\to 0} \mean_\mu\Bigl[\frac{T_{he_i} f-f}{h}\Bigr]= \lim_{h\to 0} \frac{\mean_\mu[T_{he_i} f]-\mean_\mu[f]}{h}=0.
\]
If $U\in L^2_{pot}$, we find $f_n\in \mathcal{C}$ such that $\nabla f_n\to U$ in $L^2(a)$, hence in $L^1(\Omega,\mu)^d$.
It follows
\[
\mean_\mu[U]=\lim_{n\to\infty}\mean_\mu[\nabla f_n] = 0.
\]

We now prove $(ii)$. Consider again $f\in \mathcal{C}$. Then $x\to f(x;\omega)$ is infinitely many times differentiable, $\mu$-almost surely. Integrating by parts we get
\[
\int_{\R^d} D_i f(x;\omega) \partial_j\eta(x)\,dx = -\int_{\R^d} f(x;\omega) \partial_i\partial_j\eta(x)\,dx,
\]
finally switch the partials and conclude
\[
\int_{\R^d} D_i f(x;\omega) \partial_j\eta(x)\,dx = \int_{\R^d} D_j f(x;\omega) \partial_i\eta(x)\,dx.
\]
For a general $U\in L^2_{pot}$ take approximations and use the fact that $\nabla f_n\to U$ in $L^2(a)$ implies $D_i f_n(\cdot;\omega)\to U_i(\cdot;\omega)$ in $L^1_{loc}(\R^d)$ $\mu$-almost surely.
\end{proof}

\paragraph{Weyl's decomposition} Since $L^2(a)$ is an Hilbert space and $L^2_{pot}$ is by construction a closed subspace, we can write
\[
L^2(a) = L^2_{pot}\oplus (L^2_{pot})^\perp.
\]
We want to decompose the bounded functions $\{\pi^k\}_{k=1}^d$, where $\pi^k$ is the unit vector in the kth-direction. Since $\pi_k\in L^2(a)$, for each $k=1,\dots,d$, there exist  functions $U^k\in L^2_{pot}$ and  $R^k\in  (L^2_{pot})^\perp$ such that $\pi^k=U^k+R^k$. By definition of orthogonal projection we have
\[
\mean_\mu[\langle a U^k,V \rangle]=\mean_\mu[\langle a \pi^k,V \rangle],\quad\forall V\in
L^2_{pot}.
\]
\begin{remark} By definition of $L^2_{pot}$ and orthogonal projection it follows in particular that
\[
\mean_\mu[\langle a (U^k-\pi_k),U^k-\pi_k \rangle]=\inf_{f\in\mathcal{C}}\mean_\mu[\langle a (\nabla f-\pi_k),\nabla f-\pi_k \rangle].
\]
\end{remark}

\begin{proposition}\label{prop:pos} Set $\mathbf{d}_{ij}\defeq 2\mean_\mu[\langle a (U^i-\pi_i),U^j-\pi_j \rangle]$. Then the matrix $\{\mathbf{d}_{ij}\}_{i,j}$ is positive definite.
\end{proposition}

\begin{proof} Take any $\xi\in\R^d$, then
\[
\sum_{i,j} \mathbf{d}_{ij} \xi_i \xi_j = 2 \mean_\mu\Bigl[\langle a \Bigl(\sum_i \xi_i U^i-\xi\Bigr),\sum_j\xi_j U^j-\xi \rangle\Bigr].
\]
Since $\sum_i \xi_i U^i\in L^2_{pot}$ is the orthogonal projection of the function $\pi_\xi:\omega\to\xi$,  and $\pi_\xi\in L^2(a)$, we have
\begin{align}\label{eq:cp_1}
\sum_{i,j} \mathbf{d}_{ij} \xi_i \xi_j  &=\inf_{\phi\in\mathcal{C}}2\mean_\mu[\langle a( \nabla\phi-\xi),\nabla \phi-\xi \rangle]\geq \sum_{i=1}^d\inf_{\phi\in\mathcal{C}} 2\mean_\mu[\lambda| D_i\phi-\xi_i|^2] \notag\\
&=\sum_{i=1}^d|\xi_i|^2\inf_{\phi\in\mathcal{C}}2\mean_\mu[\lambda|D_i\phi-1|^2]
\end{align}
we end up with a basic one dimensional problem. Observe that by H\"older's inequality we have
\[
\mean_\mu[\lambda|D_i\phi-1|^2] \geq \mean[\lambda^{-1}]^{-1}\mean_\mu[(D_i\phi-1)]^2 = \mean[\lambda^{-1}]^{-1}
\]
for all $\phi\in C^\infty_b(\Omega)$ since by Lemma \ref{lemma:L2pot} we have that $\mean_\mu[D_i \phi ]=0$. Therefore \eqref{eq:cp_1} is bounded from below by $\sum_{i=1}^d|\xi_i|^2\mean_\mu[\lambda^{-1}]^{-1}=|\xi|^2\mean_\mu[\lambda^{-1}]^{-1}$ and we get the bound
\[
\sum_{i,j} \mathbf{d}_{ij} \xi_i \xi_j\geq 2 \mean_\mu[\lambda^{-1}]^{-1} |\xi|^2
\]
which is what we wanted to proof.
\end{proof}

At this point we build the corrector starting from the functions $U^k\in L^2_{pot}$. For $k=1,...,d$ we define the \emph{corrector} to be the function $\chi^k :\R^d\times\Omega\to\R$ such that
\[
\chi^k(x,\omega)\defeq \sum_{j=1}^d \int_0^1 x_j U_j^k(\tau_{tx}\omega)\,dt.
\]
It is not hard to prove that $\chi^k$ is well defined, and taking expectation it follows that $\mean_\mu[\chi^k(x,\omega)]=0$. The key result about the corrector is listed here below
\begin{proposition}(Weak differentiability) For $k=1,...,d$ the function $x\to\chi^k(x,\omega)$ is in $L^1_{loc}(\R^d)$, weakly differentiable $\mu$-almost surely and $\partial_i \chi^k(x,\omega)= U_i^k(\tau_{x}\omega)$.
\end{proposition}
\begin{proof} Let $\eta\in C^\infty_0(\R^d)$ be any test function and calculate
\[
 \int_{\R^d} \chi^k(x,\omega) \partial_i \eta(x)\,dx = \int_{\R^d}\sum_{j=1}^d\int_0^1 x_j
U_j^k(\tau_{tx}\omega)\,dt\, \partial_i \eta(x)\,dx.
\]
By changing the order of integration and applying the change of variables $y=tx$ we get
\[
  \int_0^1\sum_{j=1}^d\int_{\R^d}
U_j^k(\tau_{y}\omega)\frac{y_j}{t^{d+1}}\partial_i\eta\Bigl(\frac{y}{t}\Bigr)\,dx\,dt.
\]
Next observe that for $j\neq i$,
\[
  \frac{y_j}{t^{d+1}}\partial_i\eta\Bigl(\frac{y}{t}\Bigr) =
\partial_i\Bigl(\frac{y_j}{t^{d}}\eta\Bigl(\frac{y}{t}\Bigr)\Bigr),
\]
which together with property $(ii)$ of Lemma \ref{lemma:L2pot} gives.
\[
\int\chi^k(x,\omega) \partial_i \eta(x)\,dx = \int
U_i^k(\tau_{y}\omega)\int_0^1\sum_{j\neq i}\partial_j\Bigl(\frac{y_j}{t^{d}}\eta\Bigl(\frac{y}{t}\Bigr)\Bigr)+
\frac{y_i}{t^{d+1}}\partial_i\eta\Bigl(\frac{y}{t}\Bigr)dt\,dx.
\]
Finally, observe that for $y\neq 0$
\begin{equation*}
 \int_0^1\sum_{j\neq i}\partial_j\Bigl(\frac{y_j}{t^{d}}\eta\Bigl(\frac{y}{t}\Bigr)\Bigr)+
\frac{y_i}{t^{d+1}}\partial_i\eta\Bigl(\frac{y}{t}\Bigr)dt= -\int_0^1\frac{d}{dt} \Bigl(\eta\Bigl(\frac{y}{t}\Bigr)\frac{1}{t^{d-1}}\Bigr)\,dt = -\eta(y).
\end{equation*}
This ends the proof since it follows that
\begin{equation}\label{eq:weak}
 \int_{\R^d} \chi^k(x,\omega) \partial_i \eta(x)\,dx =-\int_{\R^d} U^k_i(x;\omega) \eta(x) \,dx.
\end{equation}
One may think that the set of $\omega$ for which \eqref{eq:weak}
holds, depends on $\eta$. Since $C^\infty_0(\R^d)$ is separable we can remove such ambiguity considering a countable dense subset $\{\eta_n\}_{n\in\N}$ of $C^\infty_0(\R^d)$.
\end{proof}

So far we did not need more than the first moment for $\lambda^{-1}$ and $\Lambda$. To get more regularity and exploit the power of Sobolev's embedding theorems, we shall now assume \ref{ass:a.2}, namely, for $1/p+1/q<2/d$ we suppose that $\mean_\mu[\lambda^{-q}]$, $\mean_\mu[\Lambda^{p}]<\infty$. Such an assumption has the following consequence.
\begin{proposition}\label{prop:integrability} Assume \ref{ass:a.1} and \ref{ass:a.2}, then the corrector $\chi^k(\cdot,\omega)\in\F_{loc}^\omega$ for $\mu$-almost all $\omega\in\Omega$.
\end{proposition}
\begin{proof} By construction, there exists $\{f_n\}_{\N}\subset\mathcal{C}$ such that $\nabla f_n\to U^k$ in $L^2(a)$. This implies that for any ball $B\subset\R^d$
\[
\int_B \langle a(x;\omega)\nabla f_n(x;\omega)-\nabla \chi^k(x,\omega), f_n(x;\omega)-\nabla \chi^k(x,\omega)\rangle\,dx\to 0.
\]
Observe that $g_n(x,\omega)=f_n(x;\omega)-f_n(\omega)$ belongs to $C^\infty(\R^d)$ and satisfies
\[
g_n(x,\omega)=\sum_{i=1}^d\int_0^1 x_j\partial_j f_n(tx;\omega)\,dt.
\]
By means of \ref{ass:a.2} it is immediate to prove that $g_n\to \chi^k$ in $W^{1,2q/(q+1)}(B)$ for any ball $B\subset\R^d$. We claim that $\eta g_n\to \eta \chi^k$ in $L^2(\R^d)$ with respect to $\Ew_1$, for any cut-off $\eta$ and $\mu$-almost surely, which by definition proves $\chi^k(\cdot,\omega)\in\F_{loc}^\omega$.
Indeed
\begin{multline*}
\int_{\R^d} \langle a \nabla (\eta g_n)-\nabla (\eta\chi^k),\nabla (\eta g_n)-\nabla (\eta\chi^k)\rangle\,dx\leq\\
2\int_{B} \langle a \nabla  g_n-\nabla \chi^k,\nabla g_n-\nabla \chi^k\rangle\,dx+2\|\nabla \eta\|_\infty^2\int_{B} \Lambda |g_n-\chi^k|^2\,dx\to 0
\end{multline*}
where the last integral goes to zero by $g_n\to \chi^k$ in $W^{1,2q/(q+1)}(B)$, and by means of the Sobolev's embedding theorem $W^{1,2q/(q+1)}(B)\hookrightarrow L^{2p^*}(B)$.
\end{proof}

\subsection{Harmonic coordinates and Poisson equation}

Now that we have the corrector we want to construct a weak solution to the Poisson equation $\mathcal{L^\omega} u = 0$ for $\mu$-almost all $\omega$. Consider, for $k=1,...,d$, the \emph{harmonic coordinates} to be the functions $y^k : \R^d\times \Omega\to \R$ defined by $y^k(x,\omega)\defeq x_k-\chi^k(x,\omega)$.

We say that a function $u\in\F_{loc}$ is $\Ew$-harmonic if $\Ew(u,\phi)=0$, $\forall \phi\in C_0^\infty(\R^d)$. The next proposition justifies the name harmonic coordinates.
\begin{proposition}\label{prop:solution} For $k=1,..,d$, the harmonic coordinates $x\to y ^k(x,\omega)$ are $\Ew$-harmonic $\mu$-almost surely.
\end{proposition}
\begin{proof}
We have to prove that $\mu$-almost surely, for all $\phi\in C_0^\infty(\R^d)$
\[
\Ew(y^k,\phi)= \sum_{i,j}\int_{\R^d} a_{ij}(x;\omega)\partial_iy ^k(x,\omega)\partial_j\phi(x)\,dx = 0.
\]
By construction of the corrector, the stationarity of the environment and the fact that $T_x \mathcal{C}= \mathcal{C}$, we have that
\[
\sum_{i,j}\mean_\mu[a_{ij}(x;\omega)\partial_iy ^k(x,\omega)D_j f(\omega)] = 0,\quad \forall x\in\R^d,\forall f\in\mathcal{C}.
\]
Now fix $\phi\in C_0^\infty(\R^d)$ and integrate against it, we get that for all $f\in \mathcal{C}$
\begin{align*}
0&=\sum_{i,j}\int_{\R^d}\phi(x)\mean_\mu[a_{ij}(x;\omega)\partial_iy ^k(x,\omega)D_j f(\omega)]\,dx\\
&=\sum_{i,j}\mean_\mu\Bigl[a_{ij}(0;\omega)\partial_iy ^k(0,\omega)\int_{\R^d}D_j f(\tau_{-x}\omega)\phi(x)\,dx\Bigr]\\
&=\mean_\mu\Bigl[f(\omega)\sum_{i,j}\int_{\R^d}a_{ij}(x;\omega)\partial_iy ^k(x,\omega)\partial_j\phi(x)\,dx\Bigr].
\end{align*}
Since $\mathcal{C}\subset L^{p}(\Omega,\mu)$ for all $p\geq 1$ densely, it follows that
\begin{equation}\label{eq:poisson}
\sum_{i,j}\int_{\R^d}a_{ij}(x;\omega)\partial_iy ^k(x,\omega)\partial_j\phi(x)\,dx = 0,\quad\mu\mbox{-a.s.}
\end{equation}
this ends the proof. To be precise, one should observe that $C_0^\infty(\R^d)$ is separable, which ensures that \eqref{eq:poisson} is satisfied for all $\phi\in C_0^\infty(\R^d)$, $\mu$-almost surely.
\end{proof}

\begin{remark}
Observe that neither \ref{ass:a.2} nor \ref{ass:a.3} is used in the construction of the harmonic coordinates.
\end{remark}

\begin{remark}\label{rem:bound}
If we define $y_\epsilon^k(x,\omega)\defeq \epsilon y^k(x/\epsilon,\omega)$, then an application of the ergodic theorem yields
\begin{equation}\label{eq:bound}
\lim_{\epsilon\to 0}\int_{B_R}\langle a(x/\epsilon;\omega)\nabla_x y_\epsilon^k(x;\omega),\nabla_x y_\epsilon^k(x;\omega) \rangle\,dx = \mean_\mu[\langle a(\pi_k-U^k),\pi_k-U^k\rangle]|B_R|<\infty.
\end{equation}
which in view of \ref{ass:a.2} and the Sobolev's embedding theorem implies that
\begin{equation}
\limsup_{\epsilon\to 0}\|1_{B_R} y^k_\epsilon\|_{\rho}<\infty,
\end{equation}
where both limits hold $\mu$-almost surely.
\end{remark}

\subsection{Martingales and Harmonic coordinates}

We will assume as usual \ref{ass:a.1}, \ref{ass:a.2} and \ref{ass:a.3}.

In a situation where $L^\omega = \nabla\cdot( a^\omega \nabla\,)$ is well defined and associated to the process $X^\omega_t$, the fact that $L^\omega y(x,\omega) = 0$, would imply that $y(X_t^\omega,\omega)$ is a martingale by It\^o's formula. In our case we lack the regularity to use the theory coming from stochastic differential equations and we must rely on Dirichlet Forms technique. We know that $y^k(x,\omega)$ is $\E^\omega$-harmonic, which in a weaker sense, is analogous to say that $y^k$ is $L^\omega$-harmonic.

We will use the following theorem due to Fukushima, \cite{Fukushima1987}[ Theorem 3.1].

\begin{theorem}\label{thm:martingale}
	Fix a point $x_0$ and assume the following conditions for a process $\mathbf{N}=(Z_t,\PR_x)$ associated to $(\E,\F)$ on $L^2(\R^d,dx)$, and for a function $u:\R^d\to\R$.
	\begin{itemize}
		\item[(i)] For all $t>0$ the transition semigroup $\mathcal{P}_t$ of $\mathbf{N}$ satisfies $\mathcal{P}_t\mathds{1}_A(x_0)=0$  whenever $Cap(A)=0$.
		\item[(ii)] $u\in \mathcal{F}_{loc}$, u is continuous and $\E$-harmonic.
		\item[(iii)] Let $\nu_{\langle u \rangle}$ be the energy measure of $u$, namely the only measure such that
		\[
		\int_{\R^d} v(x)\,d\nu_{\langle u \rangle}(dx)=2\E(uv,v)-\E(u^2,v),\quad v\in C^\infty_0(\R^d).
		\]
		We assume that $\nu_{\langle u \rangle}$ is absolutely continuous with respect to the Lebesgue measure $\nu_{\langle u \rangle} = fdx$ and that the density function $f$ satisfies
		\[
		\mean_{x_0}\biggl[\int_0^t f(Z_s)\,ds\biggr]<\infty,\quad t>0.
		\]
	\end{itemize}
	Then $M_t=u(Z_t)-u(Z_0)$ is a $\PR_{x_0}$-square integrable martingale with
	\[
	\langle M \rangle_t = \int_0^t f(Z_s)\,ds,\quad t>0,\quad \PR_{x_0}\mbox{-a.s.}
	\]
\end{theorem}

We want to apply Theorem \ref{thm:martingale} to the function $u(x,\omega)=\sum_k \lambda_k y^k(x,\omega)$, being an $\E^\omega$-harmonic function, and to the minimal process $\mathbf{M}^\omega = (X^\omega_t,\PR^\omega_x)$, $x\in\R^d$.  We fix the starting point to be $x_0=0$. Some attention is required to check that every assumption of Theorem \ref{thm:martingale} is satisfied for $\mu$-almost all $\omega\in\Omega$.

By construction, since $\mathbf{M}^\omega = (X^\omega_t,\PR^\omega_x)$, $x\in\R^d$ is the minimal diffusion for almost all $\omega\in\Omega$, it follows that $\mathcal{P}_t 1_A(0)=\int_{A} p_t^\omega(0,y)\,dy=0$ whenever $Cap(A)=0$, so that $(i)$ is satisfied. Indeed $Cap(A)=0$ implies that the Lebesgue measure of $A$ is zero \cite[Page 68]{fukushima1994dirichlet}.

Assumption $(ii)$ is  satisfied for almost all $\omega$ in view of Proposition \ref{prop:solution}, Proposition \ref{prop:integrability} and \ref{ass:a.3} which assures the continuity of $x\to y^k(x,\omega)$ for $\mu$-almost all $\omega\in\Omega$ by classical results in elliptic partial differential equations with locally uniformly elliptic coefficients \cite[Gilbarg and Trudinger]{gilbarg2001elliptic}.

In order to check assumption $(iii)$ we have first to understand $\nu_{\langle u\rangle}$. According to \cite[Theorem 3.2.2]{fukushima1994dirichlet} and using the fact that $y^k$ are weakly differentiable, the density $f(x,\omega)$ of $\nu_{\langle u\rangle}$ with respect to the Lebesgue measure is given by
\begin{equation*}
f(x,\omega)=2\sum_{i,j} \partial_i u(x;\omega)\partial_j u(x;\omega)\, a_{ij}(x;\omega) = 2\sum_{k,h} \lambda_k \lambda_h \Bigl(\sum_{i,j} \partial_i y^k(x;\omega)\partial_j y^h(x;\omega)\, a_{ij}(x;\omega)\Bigr)
\end{equation*}
which we can rewrite as $f(x,\omega)= 2\langle q(x,\omega)\lambda,\lambda\rangle$, with
\begin{equation*}
q^{hk}(\omega)\defeq\sum_{i,j} \partial_i y^k(0;\omega)\partial_j y^h(0;\omega)\, a_{ij}(\omega)=\sum_{i,j} (U_i^k(\omega)-\delta_{ik})(U_j^h(\omega)-\delta_{jh})\, a_{ij}(\omega).
\end{equation*}
Next we compute, using the stationarity of the environment process
\begin{equation*}
\int_{\Omega } \mean_0^\omega\biggl[\int_0^t f(X^\omega_s;\omega)\,ds\biggr]\,d\mu =2\int_{\Omega } \mean_0^\omega\biggl[\int_0^t \langle q(\psi^\omega_s)\lambda,\lambda\rangle\,ds\biggr]\,d\mu=2t\int_{\Omega}\langle q(\omega)\lambda,\lambda\rangle\,d\mu,
\end{equation*}
which is finite by construction, since $U\in L^2(a)$. In particular $(iii)$ is satisfied. It follows the following theorem:

\begin{theorem}\label{thm:harmoniccoho} Assume \ref{ass:a.1},\ref{ass:a.2} and \ref{ass:a.3}. Then $y(X_t^\omega,\omega)$ is a $\PR^\omega_0$-square integrable martingale with covariation given by
\[
\langle y^k(X_t^\omega,\omega), y^h(X_t^\omega,\omega)\rangle_t =2 \int_0^t \sum_{i,j}a_{ij}(X^\omega_s,\omega) (\partial_i\chi^k(X^\omega_s,\omega)-\delta_{ik})(\partial_j\chi^h(X^\omega_s,\omega)-\delta_{jh})\, \,ds,
\]
for $\mu$-almost all $\omega\in\Omega$.
\end{theorem}
\begin{proof}
Above.
\end{proof}

\section{Proof of the Invariance Principle}\label{sec:proof}

In Section \ref{sec:corrector} we constructed the function $\chi,y:\R^d\times\Omega\to\R^d$ in a way that we can decompose the process $X^\omega$ as
\[
X_t^\omega = y(X_t^\omega,\omega)+\chi(X_t^\omega,\omega),
\]
in particular, we proved in Theorem \ref{thm:harmoniccoho} that $y(X_t^\omega,\omega)$ is a martingale. In order to get a quenched invariance principle for the process $X_t^{\epsilon,\omega} = \epsilon X_{t/\epsilon^2}^\omega$ we will need to prove that $\epsilon \chi(X_t^{\epsilon,\omega}/\epsilon,\omega)$ is converging to zero in law and that the quadratic variation of the martingale is converging to a constant.

As first result on the decay of the corrector as $\epsilon\to 0$ we have the following Lemma.

 \begin{lemma}\label{lem:control} For all $R>0$ and for $\mu$-almost all $\omega\in\Omega$
 \[
 \lim_{\epsilon\to0}\|y^k_\epsilon(x;\omega)-x_k\|_{2 p^*,B_R}=\lim_{\epsilon\to0}\|\chi^k_\epsilon(x;\omega)\|_{2 p^*,B_R}=0.
 \]
 \end{lemma}
\begin{proof} It is enough to show that for any $\eta\in C_0^\infty(B_R)$ we have
\[
 \lim_{\epsilon\to 0} \int_{\R^d} y^k_\epsilon(x;\omega)\eta(x)\,dx=\int_{\R^d} x_k \eta(x)\,dx.
 \]
Indeed the above property implies the weak convergence $y_\epsilon^k\rightharpoonup
 x_k$ in $L^2(B_R)$. This gives the strong convergence in $L^{2p^*}(B_R)$, because $W^{1,2q/(q+1)}(B_R)$ is compactly embedded in $L^{2p^*}(B_R)$ and the sequence $\{y_\epsilon\}_{\epsilon>0}$ is bounded in $W^{1,2q/(q+1)}(B_R)$ by \eqref{eq:bound}.

Since $\partial_j y^k(x;\omega) =\delta_{jk}- U_j^k(\tau_x\omega)$ and $\mean_\mu[U_j^k]=0$, the ergodic theorem implies that for each $\delta>0$ arbitrary, $\mu$-almost surely, there exists $\epsilon(\omega)>0$ such that for all $\epsilon,s >0$ with $s > \epsilon/\epsilon(\omega)$
 \begin{equation}\label{eq:c1}
 \biggl|\sum_j\int_{B_R} \partial_j y_\epsilon^k(sx;\omega) x_j \eta(x)\,dx-\int_{\R^d} x_k \eta(x)\,dx\biggr|\leq \delta.
 \end{equation}
 Notice that
  \begin{align}\label{eq:1}
  \int_{\R^d} y^k_\epsilon(x;\omega)\eta(x)\,dx &= \sum_j\int_{B_R}\int_0^1 \partial_j y_\epsilon^k(tx;\omega) x_j \eta(x)\, dt \,dx \notag\\
  &=\sum_j\int_0^1\int_{B_R} \partial_j y_\epsilon^k(tx;\omega) x_j \eta(x)\,dx\, dt.
  \end{align}
 We split the integral in \eqref{eq:1} as the sum
 \[
 \sum_j\int_0^{\epsilon/\epsilon(\omega)}\int_{B_R} \partial_j y_\epsilon^k(tx) x_j \eta(x)\,dx\, dt + \sum_j\int_{\epsilon/\epsilon(\omega)}^1\int_{B_R} \partial_j y_\epsilon^k(tx) x_j \eta(x)\,dx\, dt,
 \]
 now we estimate each of the two terms.
 We can rewrite the second term as
 \[
 (1-\epsilon/\epsilon(\omega))\int_{B_R} x_j \eta(x)\,dx+\int_{\epsilon/\epsilon(\omega)}^1 r_{\epsilon/t}\,dt,
 \]
 where the second integral is bounded by $\delta$, in view of \eqref{eq:c1}. For what concerns the first part, we can easily compute
 \begin{align*}
\sum_{j}\int_{0}^{\epsilon/\epsilon(\omega)}\int_{B_R} \partial_j y_\epsilon^k(tx) x_j \eta(x)\,dx = \epsilon/\epsilon(\omega) \int_{B_R} \epsilon(\omega)y^k(x/\epsilon(\omega)) \eta(x)\,dx.
 \end{align*}
 Hence the first part is bounded by $c\cdot (\epsilon/\epsilon(\omega))$ for a constant $c>0$. Finally this yields
 \[
 \limsup_{\epsilon\to 0}\biggl|\int_{\R^d} y^k_\epsilon(x;\omega)\eta(x)\,dx-\int_{\R^d} x_k \eta(x)\,dx\biggr|\leq \delta
 \]
 with $\delta$ arbitrarily chosen.
 \end{proof}

\begin{proposition}\label{prop:sublinearity}
For all $R>0$,
\begin{equation}\label{eq:sublin}
\lim_{\epsilon\to 0} \sup_{|x|\leq R} \epsilon|\chi(x/\epsilon,\omega)|  = 0,\quad \mu\mbox{-almost surely.}
\end{equation}
\end{proposition}

\begin{proof}Observe that $\chi_\epsilon^k(x,\omega)\defeq \epsilon\chi(x/\epsilon,\omega)$ is a solution on $B=B(R)$ for all $\epsilon>0$ of
\[
\sum_{i,j}\int_{B}  a_{ij}^\omega(x/\epsilon)\partial_i\chi_\epsilon^k(x;\omega)\partial_j \phi(x)\,dx=\sum_{i,j}\int_B a_{ij}^\omega(x/\epsilon)\partial_i f_k(x)\partial_j \phi(x)\,dx,
\]
where $f_k(x) = x_k$ and $\phi\in C^\infty_0(B)$. Clearly $|\nabla f_k(x)|\leq 1$ for all $x\in \R^d$ and $\epsilon>0$. By Lemma \ref{lem:control},
we get that
\[
\lim_{\epsilon\to0}\|\chi^k_\epsilon(x;\omega)\|_{2p^*,B_R}=0
\]
Therefore, we can obtain \ref{eq:sublin} applying \eqref{cor:moser} with $\alpha = 2 p^*$.
\begin{equation*}
\|\chi_\epsilon^k\|_{B(R),\infty}\leq \\ C_3
\biggl(1\vee\|(\lambda^\omega)^{-1}\|_{B(2R/\epsilon),q}\|\Lambda^\omega\|_{ B(2R/\epsilon),p} \biggr)^{\kappa'} \|\chi_\epsilon^k\|_{B(2 R),2p^*}^{\gamma'}\vee\|\chi_\epsilon^k\|_{B(2 R),2p^*}
\end{equation*}
which goes to zero as $\epsilon\to 0$ by Lemma \ref{lem:control}. Notice that we can bound $\|\lambda^{-1}\|_{B(2R/\epsilon),q}\|\Lambda\|_{ B(2R/\epsilon),p}$ by a constant, by means of \ref{ass:a.2} and the ergodic theorem.
\end{proof}

We can now turn to the proof of Theorem \ref{thm:invprinc}, namely the quenched invariance principle for the diffusions $\epsilon X^\omega_{t/\epsilon^2}$.

\begin{proof}[Proof Theorem 1.1] With the help of Proposition \ref{prop:sublinearity} the proof of this theorem is identical to \cite[Theorem 1]{fannjiang1997}, with only a minor difference, namely,
the limiting matrix $\mathbf{D}=[\mathbf{d}_{ij}]$ is given by
\[
\mathbf{d}_{ij} = 2\mean_\mu[\langle a(\omega) \nabla y^i(0,\omega),\nabla y^j(0,\omega) \rangle]
\]
being $y^i(x,\omega)$ the harmonic coordinates as constructed in Section \ref{sec:corrector}.

For completeness we put her the proof of part (ii) of the theorem and we refer to \cite{fannjiang1997} for the first part.  We make use of the decomposition
\[
\epsilon X^\omega_{t/\epsilon^2} = \epsilon y( X^\omega_{t/\epsilon^2},\omega) + \epsilon \chi( X^\omega_{t/\epsilon^2},\omega).
\]
and the fact that $M^{\epsilon,\omega} = \epsilon y( X^\omega_{t/\epsilon^2},\omega)$ is a $\PR_0^\omega$-square integrable continuous martingale $\mu$-almost surely by Theorem \ref{thm:martingale}. Its quadratic variation is given by
\[
\langle M_h^{\epsilon,\omega}, M_k^{\epsilon,\omega}\rangle_t = \epsilon \int_0^{t/\epsilon^2} 2\sum_{i,j}a_{ij}(X^\omega_s,\omega) (\partial_i\chi^k(X^\omega_s,\omega)-\delta_{ik})(\partial_j\chi^h(X^\omega_s,\omega)-\delta_{jh})\, \,ds.
\]
An application of the ergodic theorem for the environmental process shows that
\[
\lim_{\epsilon\to 0}\langle M_h^{\epsilon,\omega}, M_k^{\epsilon,\omega}\rangle_t = \mathbf{d}_{hk}t,
\]
$\PR^\omega_0$-almost surely, but also in the $L^1$ sense for almost all $\omega\in\Omega$. We can now apply the central limit for martingales \cite[Theorem 5.4]{helland1982} to conclude that the martingale $M^{\epsilon,\omega}$ converges in distribution over $C([0,\infty),\R^d)$ under $\PR_0^\omega$ to a Wiener measure with covariances given by $\mathbf{D}$. The matrix is non degenerate by Proposition \ref{prop:pos}.

It remains to show that the correctors $\epsilon \chi(X^{\omega}_{t/\epsilon^2},\omega)$ converge to zero in distribution. For  that the sublinearity of the corrector will play a major role.

Let $T>0$ be a fixed time horizon. We claim that for all $\delta>0$
\begin{equation}\label{eq:claim}
\lim_{\epsilon\to 0} \PR_0^\omega\Bigl(\sup_{0\leq t \leq T} |\epsilon \chi(X^{\omega}_{t/\epsilon^2},\omega)| >\delta \Bigr) = 0.
\end{equation}
Denote by $\tau_R^{\epsilon,\omega}$ the exit time of $\epsilon X^{\omega}_{t/\epsilon^2}$ from the ball $B$ of radius $R>1$ centered at the origin. Observe that
\begin{align*}
\limsup_{\epsilon\to 0} \PR_0^\omega&\Bigl(\sup_{0\leq t \leq T} |\epsilon \chi(X^{\omega}_{t/\epsilon^2},\omega)| >\delta \Bigr)\\
& \leq \limsup_{\epsilon\to 0} \PR_0^\omega\Bigl(\sup_{0\leq t \leq \tau_R^{\epsilon,\omega}} |\epsilon \chi(X^{\omega}_{t/\epsilon^2},\omega)| >\delta \Bigr)+\limsup_{\epsilon\to 0} \PR_0^\omega\Bigl(\sup_{0\leq t \leq T} |\epsilon X^{\omega}_{t/\epsilon^2}| > R \Bigr).
\end{align*}
\emph{First addendum}: By Proposition
\ref{prop:sublinearity}
\[
\lim_{\epsilon\to 0} \sup_{0\leq t \leq \tau_R^{\epsilon,\omega}} |\epsilon \chi(X^{\omega}_{t/\epsilon^2},\omega)| = 0.
\]
and therefore $\mu$-almost surely
\[
\limsup_{\epsilon\to 0} \PR_0^\omega\Bigl(\sup_{0\leq t \leq \tau_R^{\epsilon,\omega}} |\epsilon \chi(X^{\omega}_{t/\epsilon^2},\omega)| >\delta \Bigr)=0.
\]
\emph{Second addendum}: we use again Proposition
\ref{prop:sublinearity} to say that there exists $\bar{\epsilon}(\omega) >0$, which may depend on $\omega$ such that for all $\epsilon<\bar{\epsilon}(\omega)$ we have $\sup_{0\leq t \leq \tau_R^{\epsilon,\omega}} |\epsilon \chi(X^{\omega}_{t/\epsilon^2},\omega)|<1$. For such $\epsilon$ we have $\mu$-almost surely
\begin{align*}
\PR_0^\omega\Bigl(\sup_{0\leq t \leq T} |\epsilon X^{\omega}_{t/\epsilon^2}|\geq R\Bigr) &= \PR_0^\omega\Bigl(\tau_R^{\epsilon,\omega}\leq T\Bigr)\\
&=\PR_0^\omega\Bigl(\tau_R^{\epsilon,\omega}\leq T,\sup_{0\leq t \leq \tau_R^{\epsilon,\omega}} |\epsilon y(X^{\omega}_{t/\epsilon^2},\omega)|>R-1\Bigr)\\
&\leq\PR_0^\omega\Bigl(\sup_{0\leq t \leq T} |\epsilon y(X^{\omega}_{t/\epsilon^2},\omega)|>R-1\Bigr)
\end{align*}
Since $\epsilon y(X^{\omega}_{\cdot/\epsilon^2},\omega)$ converges in distribution under $\PR_0^\omega$ to a  non-degenerate Brownian motion with deterministic covariance matrix given by $\mathbf{D}$ we have that there exists positive constants $c_1,c_2$ independent on $\epsilon$ and $\omega$ such that
\[
\limsup_{\epsilon\to 0} \PR_0^\omega\Bigl(\sup_{0\leq t \leq T} |\epsilon y(X^{\omega}_{t/\epsilon^2},\omega)|>R-1\Bigr)\leq c_1 e^{-c_2 R},
\]
from which it follows
\[
\limsup_{\epsilon\to 0} \PR_0^\omega\Bigl(\sup_{0\leq t \leq T} |\epsilon X^{\omega}_{t/\epsilon^2}|>r\Bigr)\leq c_1 e^{-c_2 R}.
\]
Therefore
\[
\limsup_{\epsilon\to 0} \PR_0^\omega\Bigl(\sup_{0\leq t \leq T} |\epsilon \chi(X^{\omega}_{t/\epsilon^2},\omega)| >\delta \Bigr) \leq c_1 e^{-c_2 R}
\]
and since $R>1$ was arbitrary, the claim \eqref{eq:claim} follows, namely the corrector converges to zero in law under $\PR_0^\omega$, $\mu$-almost surely.

The convergence to zero in law of the correctors $\epsilon \chi(X_{\cdot/\epsilon^2},\omega)$, combined with the fact that $\epsilon y(X_{\cdot/\epsilon^2},\omega)$ satisfies an invariance principle $\mu$-almost surely and that $\epsilon X^\omega_{\cdot/\epsilon^2} =\epsilon \chi(X_{\cdot/\epsilon^2},\omega)+ \epsilon y(X_{\cdot/\epsilon^2},\omega)$,  implies that also the family $\epsilon X^\omega_{\cdot/\epsilon^2}$ under $\PR_0^\omega$ over $C([0,\infty),\R^d)$ satisfies an invariance principle $\mu$-almost surely with the same limiting law.
\end{proof}

\begin{corollary}
	Let $\theta:\Omega\to\R$ be a $\mathcal{G}$-measurable function and assume that $\theta(\tau_.\omega)$, $\theta(\tau_.\omega)^{-1}\in L^\infty_{loc}(\R^d)$ for $\mu$-almost all $\omega\in\Omega$ and that $\mean_\mu[\theta],\mean_\mu[\theta^{-1}]<\infty$. Let $\mathbf{M}^{\theta,\omega}\defeq(X_t^{\theta,\omega},\PR^{\theta,\omega}_x)$, $x\in\R^d$ the minimal diffusion process associated to $(\E^\omega,\F^{\theta,\omega})$ on $L^2(\R^d,\theta dx)$. Then, for $\mu$-almost all $\omega\in\Omega$, the laws of the processes $\epsilon X^{\theta,\omega}_{t/\epsilon^2}$ over $C([0,\infty),
	\R^d)$ converge weakly   as $\epsilon\to 0$ to a Wiener measure with covariance matrix given by $\mathbf{D}/\mean_\mu[\theta]$, where $\mathbf{D}$ was given in Theorem \ref{thm:invprinc}.
\end{corollary}
\begin{proof} Let us define the time change
\[
\hat{X}^\omega_t\defeq X_{\tau^\omega_t}^\omega,\quad\tau^\omega_t=\inf\{s>0;\,A_s^\omega\defeq\int_0^s\theta(X^\omega_u,\omega)\,du>t\}
\]	
To get asymptotic for $\epsilon^2 A_{t/\epsilon^2}$ it is easy by means of the ergodic theorem for the environmental process. We can prove as in \cite[Lemma 15]{BaMathieu} that
 \begin{equation}\label{eq:timechange}
\lim_{\epsilon\to 0}\sup_{s\in[0,t]}|\epsilon^2 A^\omega_{s/\epsilon^2}-s \mean_\mu[\theta]|=0,\quad \PR_x^\omega\mbox{-a.s}, \mbox{ a.a. }x\in\R^d,
\end{equation}
for $\mu$-almost all $\omega\in\Omega$.
Observe that $\epsilon\hat{X}^\omega_{A^\omega(t/\epsilon^2)} = \epsilon X_{t/\epsilon^2}^\omega$, then the convergence for $\epsilon\hat{X}^\omega_{t/\epsilon^2}$ $\PR_x^\omega$-a.s, for almost all $x\in\R^d$,
for $\mu$-almost all $\omega\in\Omega$ follows from Theorem 1.1 and \eqref{eq:timechange}. On the other hand the processes $\hat{X}_t^\omega$ and $X_t^{\theta,\omega}$ are equivalent, since they possess the same Dirichlet form, see Theorem 6.2.1 in \cite{fukushima1994dirichlet}. Hence the same convergence holds for $\epsilon X_{t/\epsilon^2}^{\theta,\omega}$.
\end{proof}

\end{document}